\documentclass[preprint]{article}
\usepackage{amssymb,comment,mathtools}
\usepackage{amsfonts}
\usepackage{mathrsfs}
\usepackage{amssymb,amsfonts,amsmath}
\usepackage{graphicx,graphics,epstopdf}
\usepackage{url,bigints}
\graphicspath{{./figures_cv/}}
\usepackage{latexsym,float,tikz,ulem,cancel,algorithmic,algorithm}
\usetikzlibrary{arrows.meta}
\let\oldnl\nl% Store \nl in \oldnl
\newcommand{\nonl}{\renewcommand{\nl}{\let\nl\oldnl}}% Remove line number for one line
\usepackage{adjustbox}
\usepackage{esint}
\usepackage{multirow}
\usepackage{ctable}

\usepackage{subfigure}
\usepackage{times}
\usepackage{multicol,enumerate}
\usepackage{epstopdf}
\usepackage{color}
\usepackage{fullpage}
\usepackage{setspace}
\usepackage[english]{babel}
\usepackage[version=3]{mhchem}
\usepackage{amsmath, amssymb, amsthm}
\usepackage{verbatim, latexsym}
\usepackage{colortbl}
\usepackage{multirow}
\usepackage[show]{ed}
\usepackage{hyperref}
\usepackage[OT2,OT1]{fontenc}
\newcommand\cyr
{
\renewcommand\rmdefault{wncyr}
\renewcommand\sfdefault{wncyss}
\renewcommand\encodingdefault{OT2}
\normalfont
\selectfont
}
\DeclareTextFontCommand{\textcyr}{\cyr}

\def\XXint#1#2#3{{\setbox0=\hbox{$#1{#2#3}{\int}$ }
\vcenter{\hbox{$#2#3$ }}\kern-.6\wd0}}

\newcommand{\op}{\operatorname}
\newtheorem{theorem}{Theorem}[section]
\newtheorem{corollary}{Corollary}[section]

\newtheorem{remark}{Remark}[section]
\newtheorem{lemma}{Lemma}[section]

\numberwithin{equation}{section}
\usepackage{soul,ulem}
\setstcolor{red}
\allowdisplaybreaks

\title{Robust Edge Multiscale Space based Hybrid Schwarz Preconditioner for High-Frequency Helmholtz Problems}
\author{Shubin Fu\thanks{Eastern Institute for Advanced Study, Eastern Institute of Technology, Ningbo, Zhejiang 315200, P.R. China (\texttt{shubinfu@eias.ac.cn})} \and Shihua Gong\thanks{School of Science and Engineering, The Chinese University of Hong Kong, Shenzhen, Guangdong 518172, P.R. China ({\tt{gongshihua@cuhk.edu.cn}})} \and Guanglian Li\thanks{Department of Mathematics, The University of Hong Kong, Pok Fu Lam Road, Hong Kong SAR, P.R. China ({\tt{lotusli@maths.hku.hk}}, {\tt ylliao@hku.hk})} \and Yulei Liao\footnotemark[3]}
\begin{document}
\maketitle

\begin{abstract}
We develop a novel hybrid Schwarz method, termed the robust edge multiscale space based hybrid Schwarz method (REMs-HS), for the efficient solution of high-frequency Helmholtz problems. The method combines a one-level restricted additive Schwarz preconditioner with impedance transmission conditions (RAS-imp) and an edge multiscale (EMs) coarse solver in a multiplicative way. A key ingredient of our construction is a modified impedance boundary condition (mIBC), which is imposed both on the subdomain problems and on the local problems that generate the EMs coarse space. Compared with the original impedance boundary condition, the mIBC makes the proposed method both stable and inexpensive theoretically. We prove that the resulting REMs-HS preconditioner gives rise to a uniform convergent fixed-point iteration for large wavenumbers. Moreover, the required resolution conditions are substantially relaxed. Only a mild condition on the level parameter $\ell$ of the hierarchical coarse basis $\ell>\frac{1}{2}\log_2 k$ is needed, which is markedly weaker than the conditions in previous analyses. Extensive numerical experiments demonstrate the robustness of the proposed method for Helmholtz problems with large wavenumbers. 
 \end{abstract}
 

\section{Introduction}
We consider the interior impedance problem
\begin{equation}\label{hel_cont}
\left\{\begin{aligned}
-\Delta u-k^2u&=f &&\text{in }\Omega,\\
\partial_nu+\mathrm{i}ku&=g &&\text{on }\partial\Omega,
\end{aligned}\right.
\end{equation}
where $\Omega\subset\mathbb R^d$, $d=2,3$, is a bounded convex polygonal or
polyhedral domain, $n$ is the outward normal, and $k\ge k_0>0$ is the
wavenumber. Its finite element discretization yields a large, indefinite linear system. For fixed polynomial degree, the mesh must resolve the oscillations and control the pollution error, so a fixed number of points per wavelength is generally insufficient to maintain accuracy as $k$ grows \cite{MR2812565,wu2014pre,du2015preasymptotic,GalkowskiSpence2025Preasymptotic}. The resulting growth in the number of unknowns, together with indefiniteness, motivates efficient preconditioned iterative solvers.

Several one-level domain decomposition preconditioners have been established to address these difficulties. Shifted-Laplacian methods introduce absorption to obtain an auxiliary problem that is easier to invert, while balancing this simplification against the quality of the resulting preconditioner \cite{erlangga2004class,erlangga2006novel}. Wavenumber-explicit field-of-values bounds and GMRES convergence estimates are derived for overlapping Schwarz preconditioners with local impedance conditions and absorption \cite{graham2020domain}, which are then extended to high-order finite element discretizations of heterogeneous Helmholtz problems \cite{gong2021domain}. For problems without added volume absorption, parallel overlapping Schwarz iterations with impedance transmission conditions through impedance-to-impedance maps and powers of the error operator have been proposed and analyzed in \cite{gong2022convergence}. Building on this work, the discrete RAS-imp analysis that relates finite element iterations to their continuous counterparts is established in \cite{gong2023convergence}. More recently, this one-level convergence theory has been extended to PML transmission conditions for nontrapping problems \cite{galkowski2024convergence}. Together, these contributions provide theoretical foundations for understanding how transmission conditions, discretization, and local wave propagation determine convergence under their respective assumptions.

These advances do not eliminate the need to account for global propagation as the decomposition is refined. A local correction may reduce the error within its subdomain while leaving components that require communication across many subdomains. This motivates the present extension: retaining stable local solvers while adding a coarse correction that communicates the remaining error globally. Classical two-level Schwarz theory makes this mechanism precise for positive-definite problems \cite{toselli2006domain}. Extensions to indefinite problems require suitable approximation and stability properties \cite{cai1992domain,xu1992new}. In the high-frequency setting, the size of the indefinite term makes the required coarse approximation depend explicitly on $k$. Even the invertibility of a Galerkin coarse problem needs justification. Thus, a small coarse space is useful only if its associated correction remains well defined and captures the components that local solves cannot efficiently remove.
Recent two-level analyses increasingly connect coarse-space approximation with Schwarz convergence. In particular, weighted hybrid methods based on a spectral coarse space composed of local Helmholtz-harmonic functions have been proposed, with GMRES convergence established through field-of-values estimates in the regimes considered \cite{MR4890903}. Localized orthogonal decomposition (LOD) has been used to construct coarse correctors, combined with local Dirichlet or impedance solves, in two hybrid preconditioners \cite{LuXuZhengZou2025}. Building on the Helmholtz approximation theory of \cite{MR4601686}, MS-GFEM-based two-level RAS methods have been proposed \cite{MaAlberScheichlZhang2025}, which relate Richardson and GMRES convergence to approximation errors controlled by oversampling and spectral enrichment. Moreover, \cite{GalkowskiSpence2026Hybrid} provides a general hybrid framework in which stable coarse approximation with small $L^2$-Galerkin error yields an operator close to the identity. 
Then a complementary route is established in \cite{GrahamSpence2026Polynomial} using piecewise-polynomial fine and coarse spaces with polynomial degrees growing like $\log k$ to obtain wavenumber-independent GMRES counts for an absorbing scattering formulation. 
Although the settings and operators differ across these works, they collectively identify coarse approximation as the mechanism linking local solvers to global convergence.

In this paper, we develop the robust edge multiscale space based hybrid Schwarz method (REMs-HS), using local Helmholtz problems equipped with the modified impedance boundary conditions (mIBCs).
 Edge multiscale spaces use hierarchical boundary data and local PDE extensions \cite{MR3980476,MR4860976,fu2021wavelet,FuGongLiWang:2024}. The modified impedance parameter of mIBCs equals $H^{-1}$ on artificial
subdomain boundaries and retains $k$ on the physical boundary with $H$ the coarse-mesh size. Then the corresponding local problems are coercive in the energy norm weighted by $H^{-1}$ in the regime $kH\ll1$ (Lemma \ref{lema:modified}). The same local problem is used in RAS-imp and extend continuous hierarchical boundary functions up to level $\ell$. Weighted prolongation assembles their extensions into a global coarse space. The construction uses prescribed boundary enrichment, without local spectral selection or additional volume-source basis functions due to this iIBC. We refer to Remark \ref{rmk:imp} for an explanation why this mIBC facillitates fewer multiscale basis functions. Stability of the global Galerkin coarse problem is established separately through an inf-sup argument (Lemma \ref{lem:infsup}). The coarse space targets the discrete remainder after one RAS correction which is a weighted sum of local Helmholtz-harmonic functions. Low-regularity boundary estimates (Lemma \ref{lema:vhL2} and Lemma \ref{lema:vhH1}) combining with a bound for the discrete impedance map (Lemma \ref{lema:ghH-12}) yield the required approximation properties (Lemma \ref{lema:last-app-harmonic}). Approximation of the continuous adjoint solution (Lemma \ref{lema:last}) then supplies the Schatz argument needed for the coarse projection. Theorem~\ref{thm:wavelet-basedconv} gives, in the discrete $H^1_k$-operator norm,
\begin{equation}\label{eq:intro-contraction}
\|I-B_{\mathrm{hybrid}}A\|
\lesssim \frac{2^{-\ell}H}{\delta}\sqrt{\frac{H}{h}},
\end{equation}
under certain conditions, where $h$ is the fine-mesh size and $\delta$ the minimum overlap. For a small $\varepsilon>0$, Corollary~\ref{coro:main} admits $h\simeq k^{-2-\varepsilon}$, $H\simeq\delta\simeq k^{-1-\varepsilon}$, and $\ell=\lceil(1/2+\varepsilon)\log_2k\rceil$, giving a contraction bound $\mathcal O(k^{-\varepsilon})$ for \eqref{eq:intro-contraction}. Thus the boundary hierarchy has an explicit logarithmic enrichment rule within a coupled choice of meshes and overlap. Compared with the previous theoretical result in \cite{FuGongLiWang:2024}, which required at least $\ell>2\log_2k$ levels, the present construction substantially relaxes the resolution requirements. Numerical experiments assess this scaling, the broader regimes $h=k^{-3/2}$ and fixed points per wavelength. Comparisons of unaccelerated iteration and FGMRES(20) identify the effect of the coarse space itself. We demonstrate three-dimensional tests with more than $135$ million unknowns involve only four or five preconditioner applications. These experiments provide evidence beyond the sufficient theoretical conditions, while the convergence theorem retains the fine-mesh and overlap requirements stated above.

The remainder of this paper is organized as follows. We present in Section~\ref{sec:prep} the finite element formulation of \eqref{hel_cont} and {\em a priori} estimate. Then we introduce in Section~\ref{sec:hybrid-2} the REMs-HS method, including the novel mIBC in Section~\ref{subsec:modified} and its coarse space in Section~\ref{sec:coarse-space}. Its convergence and numerical experiments are presented in Section~\ref{sec:converge} and Section~\ref{sec:numer}, respectively. Finally, we conclude our paper with Section \ref{sec:conclusion}. Throughout, spaces are complex valued, $(v,w)_D:=\int_Dv\overline w\,\mathrm dx$, and $\|v\|_{H^1_k(D)}^2:=\|\nabla v\|_{L^2(D)}^2+k^2\|v\|_{L^2(D)}^2$. The symbols $\Re$, $\Im$, and $\overline{\cdot}$ denote real part, imaginary part, and conjugation. We write $a\lesssim b$ with constants independent of the varying wavenumber and mesh parameters under the stated assumptions, $a\simeq b$ for two-sided bounds, and $a\ll b$ for a sufficiently small ratio.

\section{The Helmholtz equation and its finite element approximation}\label{sec:prep}
We summarize the basic properties of \eqref{hel_cont} in this section. To begin with, the variational formulation for \eqref{hel_cont} reads: Find $u\in H^1(\Omega)$ such that 
\begin{equation}
\label{var_form}
    a(u,v)=(f,v)_\Omega+(g,v)_{\partial\Omega}, \quad\forall v\in H^1(\Omega),
\end{equation}
where the sesquilinear form $a(\cdot,\cdot): H^1(\Omega)\times H^1(\Omega)\to \mathbb{C}$ is defined by
\[
a(v,w):=(\nabla v,\nabla w)_\Omega-k^2(v,w)_\Omega+\mathrm{i}k(v,w)_{\partial\Omega} \quad\forall v,w\in H^1(\Omega).
\]
The existence and uniqueness of the weak solution of \eqref{var_form} and several basic properties can be found in \cite[\S 8.1]{melenk1995generalized}, which are recapitulated in the following. 
\begin{lemma}\label{lema:sesquilinear}
The following properties hold.
\begin{enumerate}
\item The sesquilinear form $a(\cdot,\cdot)$ is bounded: 
    \begin{align}\label{eq:sesqui-bound}
    |a(v,w)|\lesssim\|v\|_{H^1_k(\Omega)}\|w\|_{H^1_k(\Omega)}, \quad\forall v,w\in H^1(\Omega).
    \end{align}
\item G\aa rding's inequality holds
\begin{equation}\label{eq:garding}
\Re a(v,v)+2k^2\|v\|_{L^2(\Omega)}^2\geq \|v\|_{H^1_k(\Omega)}^2,\quad\forall v\in H^1(\Omega).
\end{equation}
\item Assume that $f\in L^2(\Omega)$ and $g\in L^2(\partial\Omega)$. Let $u\in H^1(\Omega)$ be the solution of \eqref{hel_cont}, then
\begin{equation}\label{eq:uH1}
\|u\|_{H^1_k(\Omega)}\lesssim \|f\|_{L^2(\Omega)}+\|g\|_{L^2(\partial\Omega)}.
\end{equation}
Moreover, if $g\in H^{1/2}(\partial\Omega)$, then
\begin{equation}\label{eq:uH2}
\|u\|_{H^2(\Omega)}\lesssim k\Bigl(\|f\|_{L^2(\Omega)}+\|g\|_{L^2(\partial\Omega)}\Bigr)+\|g\|_{H^{1/2}(\partial\Omega)}.
\end{equation}
\end{enumerate}
\end{lemma}
Next, we introduce the finite element discretization for problem \eqref{hel_cont}. Let $\mathcal{T}_h$ be a quasi-uniform discretization of $\Omega$ with a mesh size $h$ and let $V_h$ be the $H^1$-conforming piecewise linear finite element space
\[
V_h:=\{v\in H^{1}(\Omega): v|_{T}\in \mathbb{P}_{1}(T),\;\forall T\in \mathcal{T}_h\}
\]
with $\mathbb{P}_1(T)$ the space of linear polynomials on the fine element $T\in \mathcal{T}_h$. Then the $H^1$-conforming Galerkin approximation to problem \eqref{var_form} reads as seeking $u_h\in V_h$ such that
\begin{equation}\label{eq:discrete-global}
    a(u_h,v_h)=(f,v_h)_\Omega+(g,v_h)_{\partial\Omega}, \quad\forall v_h\in V_h.
\end{equation}
The extension to higher-order finite elements is straightforward. Let $\Pi_h:H^1(\Omega)\to V_h$ be the standard nodal interpolation operator. Then there holds
\begin{equation}\label{approxi}
   \|v-\Pi_hv\|_{L^2(T)}+h\|\nabla (v- \Pi_hv)\|_{L^2(T)}\lesssim h^2\|v\|_{H^2(T)},\quad\forall T\in\mathcal{T}_h, v\in H^2(T).
\end{equation}
Using property \eqref{approxi}, we can establish the following quasi-optimality results of the discrete problem \eqref{eq:discrete-global} \cite[Proposition 8.2.7]{melenk1995generalized}. 
\begin{lemma}[Quasi-optimality]
\label{lema_1}
Under the mesh resolution condition that $k^2h\ll 1$, problem \eqref{eq:discrete-global} is well-posed and
\begin{equation}\label{eq:uerrH1}
\|u-u_{h}\|_{H^1_k(\Omega)}\lesssim\inf_{v_h\in V_h}\|u-v_h\|_{H^1_k(\Omega)} \lesssim h\|u\|_{H^2(\Omega)}.
\end{equation}
\end{lemma}

\begin{remark}\label{rm:mesh-resolution}
The mesh resolution condition in Lemma \ref{lema_1} guarantees that the numerical solution $u_h$ is accurate enough for large $k$. It can be relaxed if $H^1$-conforming piecewise polynomial elements of order $p$ are used with $p$ large, or if only the existence of numerical solutions is required. Under the assumption that the stability of the Helmholtz problem is independent of $k$ as \eqref{eq:uH1}, the Galerkin method based on the $h$-version of FEM is quasi-optimal when $k^{p+1}h^p$ is sufficiently small \cite{melenk1995generalized,sauter2006refined,melenk2010convergence,MR2812565}. Under the weaker condition that $k^{2p+1}h^{2p}$ is small enough, the numerical solution exists and the numerical error is bounded by the data uniformly with respect to $k$ \cite{du2015preasymptotic}. In all these cases with quasi-uniform meshes, the dimension of the piecewise polynomial function space exceeds $k^{d(1+\frac{1}{2p})}$.
\end{remark}

\section{The hybrid Schwarz method}\label{sec:hybrid-2}
We present in this section the robust edge multiscale space based hybrid Schwarz method (REMs-HS). We first recap standard notation for domain decomposition methods. Let 
$\mathcal{T}_H \coloneqq \{K_i:i=1,\dots, N\}$ be a partition of $\Omega$ into a quasi-uniform mesh with mesh size $H$, of which $\mathcal{T}_h$ is a refinement. Here $N$ is the number of coarse elements. Let $\Omega_{i}$ be the smallest subdomain associated with $K_{i}\in \mathcal{T}_H$, such that
$\bigcup\left\{T\in\mathcal{T}_h:\text{dist}(T,K_{i})\leq\delta_{i}\right\}\subset \Omega_i$ and the shape of $\Omega_i$ is similar to that of $K_i$. Furthermore, let $\Gamma_i\coloneqq \partial\Omega_i\cap\partial\Omega$ and $\gamma_i\coloneqq \partial\Omega_i\backslash\partial\Omega$ be the local exterior boundary and the local interior boundary of $\Omega_i$. Note that $\gamma_i\neq\varnothing$ by this construction.
Denote by $\delta\coloneqq \min\{\text{dist}(\partial K_i\backslash\partial\Omega,\gamma_i):i=1,\cdots,N\}$ the smallest value. Then $\{\Omega_{i}\}_{i=1}^N$ is an overlapping partition of $\Omega$, i.e., 
$\overline{\Omega}=\bigcup_{i=1}^N\overline{\Omega}_{i}$. We denote by $d_i\coloneqq\op{diam}(\Omega_i)$ the diameter of $\Omega_i$ and  assume that the overlaps are small in the sense that $\delta_i\le H$ and $d_i\simeq H$ for $i=1,\dots,N$. 

We then make the {\it finite-overlap assumption} that $\max\{\#\Lambda_{\tau}:\tau\in\mathcal{T}_h\}\lesssim 1$ with $\Lambda_{\tau}=\{j:\tau\subset\overline{\Omega}_{j}\}$. Then for $v_i\in H^1(\Omega_i)$ with $\sum_{i=1}^N v_i\in H^1(\Omega)$, there holds
\begin{equation}\label{eq:sumpou-ieq}
\biggl\|\sum_{i=1}^N v_i\biggr\|_{H^1_k(\Omega)}^2 \lesssim\sum_{i=1}^N \left\| v_i\right\|_{H^1_k(\Omega_i)}^2.    
\end{equation}
Next, we introduce a sequence of discrete real-valued partition of unity functions $\{\chi_{i}\}_{i=1}^N\subset V_h$ subordinate to the cover $\{\Omega_{i}\}_{i=1}^N$, which satisfies $\sum_{i=1}^N\chi_{i}=1$ in $\overline{\Omega}$ and
\begin{equation}\label{eq:pum-property}
\operatorname{supp}(\chi_i)\subset\overline{\Omega}_{i},\qquad \| \chi_{i}\|_{L^{\infty}(\Omega_{i})}+\delta\|\nabla\chi_i\|_{L^\infty(\Omega_i)}\lesssim 1.
\end{equation}

\subsection{The modified impedance boundary condition}\label{subsec:modified}
We then introduce a novel local problem equipped with a modified impedance boundary condition (mIBC), which reads as seeking $v$ such that
\begin{equation}\label{eq:modified}
\left\{\begin{aligned}
    -\Delta v-k^2v&=f_i && \text{in }\Omega_i\\
    \partial_nv+\mathrm{i}\kappa_i v&=g_i && \text{on }\partial\Omega_i
\end{aligned}\right.
\end{equation}
with the modified impedance parameter being $\kappa_i=k$ on $\Gamma_i$ and $\kappa_i=H^{-1}$ on $\gamma_i$.

The variational formulation for \eqref{eq:modified} reads: Find $v\in H^1(\Omega_i)$ such that 
\begin{equation}
\label{eq:var_form_modified}
    a_i(v,w)=(f_i,w)_{\Omega_i}+(g_i,w)_{\partial\Omega_i}, \quad\forall w\in H^1(\Omega_i),
\end{equation}
where the local sesquilinear form $a_{i}(\cdot,\cdot)$ on $H^1(\Omega_i) \times H^1(\Omega_i)$ is given by
\begin{align}\label{eq:mod-sesqui}
a_i(v,w)\coloneqq(\nabla v,\nabla w)_{\Omega_i}-k^2(v,w)_{\Omega_i}+\mathrm{i}(\kappa_iv,w)_{\partial\Omega_i} \qquad\forall v,w\in H^1(\Omega_i).
\end{align}
Note that we mainly focus on the regime $kH\ll 1$ to make subdomains resolve the wavelength \cite{MR4621835,Graham-geneo-2024,FuGongLiWang:2024,MR4890903}, for which the problem is coercive and the solution has better regularity. We define the natural energy norm for \eqref{eq:modified} as
\begin{equation}\label{def:local-energy}
\|v\|_{H^1_{1/H}(D)}^2\coloneqq\|\nabla v\|_{L^2(D)}^2+H^{-2}\|v\|_{L^2(D)}^2,\qquad\forall D\subset\Omega_i,v\in H^1(D).
\end{equation}
We then present the basic properties for $a_i(\cdot,\cdot)$ \eqref{eq:mod-sesqui} and the {\em a priori} estimate to Problem \eqref{eq:modified}.
\begin{lemma}\label{lema:modified}
Assume that $kH\ll 1$. Then the following properties hold.
\begin{enumerate}
\item The sesquilinear form $a_i(\cdot,\cdot)$ is bounded: 
    \begin{align}\label{eq:bound_modified}
    |a_i(v,w)|\lesssim \|v\|_{H^1_{1/H}(\Omega_i)}\|w\|_{H^1_{1/H}(\Omega_i)}, \qquad\forall v,w\in H^1(\Omega_i).
    \end{align}
\item The sesquilinear form $a_i(\cdot,\cdot)$ is coercive: 
\begin{align}\label{eq:coercive_modified}
|a_i(v,v)|\gtrsim \|v\|_{H^1_{1/H}(\Omega_i)}^2,\qquad\forall v\in H^1(\Omega_i).
\end{align}
\item Assume that $f_i\in L^2(\Omega_i)$ and $g_i\in L^2(\partial\Omega_i)$. Let $v\in H^1(\Omega_i)$ be the solution of \eqref{eq:modified}, then
\begin{equation}\label{eq:uH1_modified}
\|v\|_{H^1_{1/H}(\Omega_i)}\lesssim H\|f_i\|_{L^2(\Omega_i)}+\sqrt{H}\|g_i\|_{L^2(\partial\Omega_i)}.
\end{equation}
Moreover, if $g_i\in H^{1/2}(\partial\Omega_i)$, then
\begin{equation}\label{eq:uH2_modified}
\|v\|_{H^2(\Omega_i)}\lesssim\|f_i\|_{L^2(\Omega_i)}+H^{-1/2}\|g_i\|_{L^2(\partial\Omega_i)}+\|g_i\|_{H^{1/2}(\partial\Omega_i)}.
\end{equation}
\end{enumerate}
\end{lemma}

\begin{proof}
By definition,
\[
|a_i(v,w)|\le\|v\|_{H^1_k(\Omega_i)}\|w\|_{H^1_k(\Omega_i)}+H^{-1}\|v\|_{L^2(\partial\Omega_i)}\|w\|_{L^2(\partial\Omega_i)}\lesssim\|v\|_{H^1_{1/H}(\Omega_i)}\|w\|_{H^1_{1/H}(\Omega_i)},
\]
where we have used the trace inequality
\begin{equation}\label{eq:trace}
\|v\|_{L^2(\partial\Omega_i)}\lesssim H^{-1/2}\|v\|_{L^2(\Omega_i)}+\|v\|_{L^2(\Omega_i)}^{1/2}\|\nabla v\|_{L^2(\Omega_i)}^{1/2},\qquad\forall v\in H^1(\Omega_i),
\end{equation}
in the last step. This proves \eqref{eq:bound_modified}.

To prove the coercivity of $a_i(\cdot,\cdot)$, we first recall the scaled Poincar\'e--Friedrichs inequality owing to the fact that $\gamma_i\neq\varnothing$:
\[
\|v\|_{L^2(\Omega_i)}\lesssim H\|\nabla v\|_{L^2(\Omega_i)}+\sqrt{H}\|v\|_{L^2(\gamma_i)},\qquad\forall v\in H^1(\Omega_i).
\]
Then combining with the definition \eqref{eq:mod-sesqui} with $w=v$ and the elementary calculation that $|a+\mathrm{i}b|\geq \frac{1}{2}(a+b)$ for $a,b\in\mathbb{R}$, we derive 
\begin{align*}
|a_i(v,v)|&\ge\frac{1}{2}\Bigl(\|\nabla v\|_{L^2(\Omega_i)}^2-k^2\|v\|_{L^2(\Omega_i)}^2+H^{-1}\|v\|_{L^2(\gamma_i)}^2\Bigr)\\
&\gtrsim(1-k^2H^2)\Bigl(\|\nabla v\|_{L^2(\Omega_i)}^2+H^{-1}\|v\|_{L^2(\gamma_i)}^2\Bigr)\gtrsim\|v\|_{H^1_{1/H}(\Omega_i)}^2.
\end{align*}
This proves \eqref{eq:coercive_modified}.

Next, testing \eqref{eq:var_form_modified} with $w=v$ gives
\[
|a_i(v,v)|\le\|f_i\|_{L^2(\Omega_i)}\|v\|_{L^2(\Omega_i)}+\|g_i\|_{L^2(\partial\Omega_i)}\|v\|_{L^2(\partial\Omega_i)}\lesssim\|v\|_{H^1_{1/H}(\Omega_i)}\Bigl(H\|f_i\|_{L^2(\Omega_i)}+\sqrt{H}\|g_i\|_{L^2(\partial\Omega_i)}\Bigr),
\]
where we have used the trace inequality \eqref{eq:trace} and the definition of the natural energy norm \eqref{def:local-energy} in the last step. Combining this with \eqref{eq:coercive_modified} yields \eqref{eq:uH1_modified}.

Finally, by scaling the {\em a priori} estimate of \cite[Theorem 2.3.3.2]{Grisvard:1985} for the convex polygon domain \cite[Theorem 2.4.3, Remark 2.4.5]{Grisvard:1992},
\[
\|\nabla^2w\|_{L^2(\Omega_i)}\lesssim\|\Delta w\|_{L^2(\Omega_i)}+\|\partial_nw\|_{H^{1/2}(\partial\Omega_i)}+H^{-1}\|\nabla w\|_{L^2(\Omega_i)},\qquad\forall w\in H^2(\Omega_i).
\]
Together with the equation \eqref{eq:modified} satisfied by $v$ and the trace inequality, this yields
\begin{align*}
\|\nabla^2v\|_{L^2(\Omega_i)}&\lesssim \|f_i+k^2v\|_{L^2(\Omega_i)}+\|g_i-\mathrm{i}\kappa_iv\|_{H^{1/2}(\partial\Omega_i)}+H^{-1}\|\nabla v\|_{L^2(\Omega_i)}\\
&\lesssim H^{-1}\|v\|_{H^1_{1/H}(\Omega_i)}+\|f_i\|_{L^2(\Omega_i)}+\|g_i\|_{H^{1/2}(\partial\Omega_i)},
\end{align*}
which, together with \eqref{eq:uH1_modified}, gives \eqref{eq:uH2_modified}.
\end{proof}
Note that the coercivity \eqref{eq:bound_modified} and boundedness \eqref{eq:coercive_modified} of $a_i(\cdot,\cdot)$ implies that the weak formulation \eqref{eq:var_form_modified} admits a unique solution due to the Lax--Milgram theorem. We then discretize problem \eqref{eq:var_form_modified} with the local $H^1$-conforming piecewise linear finite element space
$V_{h,i}:=\{v_h|_{\Omega_{i}}: v_h\in V_h\}\subset H^1(\Omega_i)$, which reads as seeking $v_h\in V_{h,i}$ such that 
\begin{equation}\label{eq:vh}
    a_i(v_h,w_h)=(f_i,w_h)_{\Omega_i}+(g_i,w_h)_{\partial\Omega_i}, \qquad\forall w_h\in V_{h,i}.
\end{equation}
This discrete problem is again well-posed. Its error is governed by the following standard estimate, which follows from C\'ea's lemma and an Aubin--Nitsche duality argument.

\begin{lemma}\label{lema:uerrH1_modified}
Let $v\in H^1(\Omega_i)$ and $v_h\in V_{h,i}$ be the solutions of \eqref{eq:var_form_modified} and \eqref{eq:vh}, respectively. Under the assumption $kH\ll 1$, there holds
\[
\|v-v_h\|_{L^2(\Omega_i)}+h\|v-v_h\|_{H^1_{1/H}(\Omega_i)} \lesssim h^2\|v\|_{H^2(\Omega_i)}.
\]
\end{lemma}

\begin{proof}
Since $a_i(\cdot,\cdot)$ is bounded and coercive, then C\'ea's lemma gives
\[
\|v-v_h\|_{H^1_{1/H}(\Omega_i)}\lesssim\|v-\Pi_hv\|_{H^1_{1/H}(\Omega_i)}\lesssim h\|v\|_{H^2(\Omega_i)}.
\]

Next, we consider the adjoint problem of \eqref{eq:var_form_modified}, which seeks $z\in H^1(\Omega_i)$, s.t., 
\[
a_i(w,z)=(w,v-v_h)_{\Omega_i}\qquad\forall w\in H^1(\Omega_i).
\]
Choosing $w=v-v_h$, then a combination of the Galerkin orthogonality and the boundedness of $a_i(\cdot,\cdot)$ \eqref{eq:bound_modified} yields
\begin{align*}
\|v-v_h\|_{L^2(\Omega_i)}^2&=a_i(v-v_h,z)=a_i(v-v_h,z-\Pi_hz)\lesssim\|v-v_h\|_{H^1_{1/H}(\Omega_i)}\|z-\Pi_hz\|_{H^1_{1/H}(\Omega_i)}\\
&\lesssim h^2\|v\|_{H^2(\Omega_i)}\|z\|_{H^2(\Omega_i)}\lesssim h^2\|v\|_{H^2(\Omega_i)}\|v-v_h\|_{L^2(\Omega_i)},
\end{align*}
where the adjoint $H^2$-regularity for $z$ \eqref{eq:uH2_modified} was used in the last step. This proves the desired estimate.
\end{proof}

\subsection{RAS-imp preconditioner}
We formulate in this section the one-level Schwarz domain decomposition method in operator form. Define the linear operator ${A}:V_{h}\to V_{h}'$ by
\begin{align}\label{eq:Ah-linear}
\langle {A}v_h,w_h\rangle\coloneqq a(v_h,w_h),\qquad \forall v_h, w_h \in V_{h},
\end{align}
where $\langle \cdot,\cdot\rangle$ denotes the duality pairing. The global linear system is then $A u_h=F_h$, where $F_h\in V_h'$ is defined by $\langle F_h,v_h\rangle := (f,v_h)_\Omega+(g,v_h)_{\partial\Omega}$ for all $v_h\in V_h$.

Next, we consider the local operator $A_i :V_{h,i}\to V_{h,i}'$ by
\[
\langle {A}_{i}v_h,w_h\rangle\coloneqq a_{i}(v_h,w_h),\qquad \forall v_h,w_h\in V_{h,i},
\]
where the sesquilinear form $a_i(\cdot,\cdot): H^1(\Omega_i)\times H^1(\Omega_i)\to \mathbb{C}$ is defined by \eqref{eq:mod-sesqui}.
 
Note that the functions in $V_h$ and $V_{h,i}$ are uniquely determined by their nodal values at the nodes $\{x_j: j\in \mathcal{I}\}$ in $\overline{\Omega}$ and $\{x_j: j\in \mathcal{I}_i\}$ in $\overline{\Omega}_i$, respectively. We introduce the prolongation ${R}_i^T: V_{h,i} \to V_h$ by 
\begin{align*}
R_i^Tv_h(x_j)=\left\{
\begin{aligned}
v_h(x_j), \quad j\in \mathcal{I}_i,\\
0,\quad \text{otherwise},
\end{aligned}\right.\qquad\forall v_h\in V_{h,i},
\end{align*} 
which is defined nodewise. We then define $R_i: V'_{h}\to V'_{h,i}$ by duality. To make the prolongation more stable, we employ the weighted prolongation $\tilde{R}_i^T:V_{h,i}\to V_h$ defined with the partition of unity functions by $\tilde{R}_i^T v_h\coloneqq {R}_i^T \Pi_h(\chi_i v_h)$. This weighted prolongation fulfills the property
\begin{align}\label{eq:pou}
\sum_{i=1}^N \tilde{R}_i^T(v_h|_{\Omega_i})=v_h, \quad\forall v_h\in V_h.
\end{align}
Then the restricted additive Schwarz preconditioner with mIBCs (RAS-imp) is the operator $B_{\operatorname{RAS-imp}}: V_{h}'\to V_h$ defined by $B_{\operatorname{RAS-imp}} \coloneqq \sum_{i=1}^N \tilde{R}_i^T  A_i^{-1}R_i$. The preconditioned operator $B_{\operatorname{RAS-imp}} A :V_h\to V_h$ can be written as $B_{\operatorname{RAS-imp}}A = \sum_{i=1}^N  \tilde{R}_i^T Q_i$, where the local Schwarz operator $Q_i:V_h\to V_{h,i}$ is defined by
\[
a_i(Q_i v_h, w_h) = a(v_h,R_i^T w_h),\qquad \forall v_h\in V_h\text{ and } w_h \in V_{h,i}.
\]
Finally, we end this section with an auxiliary result that demonstrates the stability of the weighted prolongation $\tilde{R}_i^T$.
\begin{lemma}\label{lema:pum-2}
Let $v_i\in V_{h,i}$ for $i=1,\dots,N$ and
$v_h:=\sum_{i=1}^N \tilde{R}_i^T v_i\in V_h$.
Then
\begin{align}
\|v_h\|_{H^1_k(\Omega)}^2\lesssim\sum_{i=1}^N \|\chi_iv_i\|_{H^1_k(\Omega_i)}^2
.\label{eq:glo-err-pum3}
\end{align}
\end{lemma}
\begin{proof}
From the approximation property \eqref{approxi} of $\Pi_{h}$ and the inverse inequality,
\begin{equation}\label{eq:Pihstability}
\|\nabla^j[(I-\Pi_h)(\chi_iv_i)]\|_{L^2(T)}\lesssim\|\nabla^j(\chi_iv_i)\|_{L^2(T)},\qquad\forall T\in\mathcal{T}_h,\quad j=0,1.
\end{equation}
Therefore, by the triangle inequality, $\tilde{R}_i^T$ is quasi-stable in the sense that
\[
\|\tilde{R}_i^Tv_i\|_{H^1_k(\Omega_i)}\le\|\chi_iv_i\|_{H^1_k(\Omega_i)}+\|(I-\Pi_h)(\chi_iv_i)\|_{H^1_k(\Omega_i)}\lesssim\|\chi_iv_i\|_{H^1_k(\Omega_i)}.
\]
Combining \eqref{eq:sumpou-ieq} with the above stability, we obtain
\[
\|v_h\|_{H^1_k(\Omega)}^2\lesssim\sum_{i=1}^N\|\tilde{R}_i^Tv_i\|_{H^1_k(\Omega_i)}^2\lesssim\sum_{i=1}^N\|\chi_iv_i\|_{H^1_k(\Omega_i)}^2,
\]
which proves the assertion.
\end{proof}
\subsection{Coarse space}\label{sec:coarse-space}
To define the coarse space, we first introduce discrete linear space $V_{h,i}^\partial$ composed of discrete Helmholtz-harmonic functions,
\begin{align}\label{eq:local-space}
V^\partial_{h,i}\coloneqq\{v_h\in V_{h,i}: a_i(v_h,w_h)=0,\quad \forall w_h\in V_{h,i}^0\},
\end{align}
where $V_{h,i}^0\subset V_{h,i}$ denotes the subspace of functions vanishing on the inner boundary,
\[
V_{h,i}^0:=\{v_h\in V_{h,i}:\;v_h=0\text{ on }\gamma_i\}.
\]
Using the weighted prolongation $\tilde{R}_i^T: V_{h,i}\to V_h$, we assemble these local spaces into the global space
\begin{align}\label{eq:glo-edge-space}
V^\partial_h:=\mathrm{span}\Bigl\{\sum_{i=1}^N \tilde{R}_i^T v_i: v_i\in V^\partial_{h,i}\Bigr\}.
\end{align}
The purpose of the coarse space $V_{h,0}$ is to approximate $V^\partial_h$, which is motivated by the fact that the local residual $v_h|_{\Omega_i}-Q_iv_h\in V_{h,i}^\partial$ for every $v_h\in V_h$. The key step is to find, in each coarse subdomain $\Omega_i$, local multiscale basis functions that approximate the local space $V^\partial_{h,i}$ well. Then the global multiscale basis functions, obtained from the local ones via the weighted prolongation $\tilde{R}_i^T$, inherit these approximation properties by Lemma \ref{lema:pum-2}. For simplicity, we describe the construction for $d=2$. The extension to $d=3$ is straightforward.

Let $\mathcal{T}_h(\Omega_i)\coloneqq \mathcal{T}_h\cap\Omega_i$ be the restriction of $\mathcal{T}_h$ on $\Omega_i$. First, we define the edge space over the boundary $\partial\Omega_i$. Let $\Gamma_i^j$, $j=1,\dots,m_i$ be a uniformly shape-regular decomposition of $\partial\Omega_i$ into affine edges with $m_i$ uniformly bounded independent of $h$ and $H$: $\cup_{j=1}^{m_i}\overline{\Gamma}_{i}^j=\partial\Omega_i$ and $\Gamma_{i}^j\cap \Gamma_{i}^{j'}=\varnothing$ for $j\neq j'$. We assume each edge $\Gamma_i^j$ contains $2^{L_i^j}$ number of fine edges $\mathcal{E}_h(\Omega_i)\coloneqq \cup_{T\in\mathcal{T}_h(\Omega_i)}\partial T$ for some $L_i^j\in\mathbb{N}$. Let $L\coloneqq\min\{L_i^j:i=1,\cdots, N,\;j=1,\cdots,m_i\}$ be the maximal level. Then there exists a quasi-uniform mesh $\mathcal{T}_{h_{\ell}}(\Omega_i)$ on $\Omega_i$ with mesh size $h_{\ell}\simeq d_i 2^{-\ell}$ for $\ell\in\mathbb{N}$ and $\ell\leq L$ such that $\mathcal{T}_h(\Omega_i)$ is its refinement. We denote $\mathcal{V}_{h_{\ell}}(\Omega_i)$ the collection of vertices and $V_{\ell}(\Gamma_{i}^j)$ the linear space spanned by the nodal bases determined by the nodes $\mathcal{V}_{h_{\ell}}(\Omega_i)\cap\Gamma_{i}^{j}$, then we define the local edge space,
\begin{align}\label{eq:local-edge-space}
E_{\ell,i}\coloneqq\Bigl\{\psi_h\in C(\partial\Omega_i):\ \psi_h|_{\Gamma_{i}^j}\in V_{\ell}(\Gamma_{i}^j)\ \ \forall\, j=1,\dots,m_i\Bigr\}.
\end{align}
We next introduce $\mathcal{P}_{\ell,i}:V_h|_{\partial\Omega_i}\to E_{\ell,i}$ as the $L^2$-projection. Then the standard $L^2$-projection error estimates \cite{MR520174} combined with an interpolation argument implies
%\begin{proposition}[Approximation properties of the hierarchical space $V_{\ell}$]\label{prop:approx-wavelets}
%Let $d=2,3$ and $s\in(0,2]$, and let $\mathcal{P}_\ell: H^s(I^{d-1})\to V_{\ell}^{\otimes^{d-1}}$ be the $L^2$-projection at level $\ell\geq 0$. Then
\begin{align}\label{eq:approx-wavelets}
%P_{\ell+1} v&=P_{\ell}v+\sum\limits_{j=0}^{2^{\ell}-1}(v,\psi_{\ell+1,j})_{I}\psi_{\ell+1,j}&\text{for all }v\in L^2(I)\\
\|v-\mathcal{P}_{\ell,i} v\|_{L^2(\partial\Omega_i)}&\lesssim 2^{-s\ell}d_i^s|v|_{H^{s}(\partial\Omega_i)},\qquad\forall v\in H^s(\partial\Omega_i) \text{ with }0\leq s\leq 2.
\end{align}
%\end{proposition}
%In particular, every corner of $\Omega_i$, being the common endpoint of two adjacent coarse edges, is assigned a single degree of freedom: the nodal basis of $E_{\ell,i}$ is obtained by taking the union, rather than the sum, of the nodal bases of the coarse edges $\Gamma_i^j$. %For $d=2$ this yields $n_{i,\ell}:=4\,(2^{\ell}+1)-4=2^{\ell+2}$ since each of the $2d$ coarse edges contributes $2^{\ell}+1$ nodal values and the $2d$ corners are otherwise counted twice. The same counting rule applies for $d=3$, where corners and edges of the coarse faces are assigned a single degree of freedom with $n_{i,\ell}:=6(2^{\ell}+1)^2-12(2^\ell+1)+8=3\times 2^{2\ell+1}+2$.
%Furthermore,
%we denote by $E_{\ell,i}^j\subset C(\gamma_i)$ the vector space spanned by hierarchical bases up to level $\ell$ on each coarse edge $\Gamma_{i}^{j}$ and continuous over $\gamma_i$. 
%The local edge space $E_{\ell,i}$ defined over $\gamma_i$ is the smallest linear space having $E_{\ell,i}^j$ as a subspace. 
Let $\{\psi_{\ell,i}^k\}_{k=1}^{n_{i,\ell}}$ and 
$\{x_{\ell,i}^k\}_{k=1}^{n_{i,\ell}}$ be the nodal bases and the associated nodal points for $E_{\ell,i}$, then the local edge space $E_{\ell,i}$ can by represented by $E_{\ell,i}= \text{span} \big\{\psi_{\ell,i}^k : 1 \leq k \leq  n_{i,\ell}\big\}$. %We depict in Figure \ref{fig:grid-points-local} the grid points $\{x_{\ell,i}^k\}_{k=1}^{n_{i,\ell}}$ over $\partial\Omega_i$ for level parameter $\ell\in\{0,1,2\}$. 
Next, we define the local multiscale space over each coarse subdomain $\Omega_i$ by 
\begin{align}\label{eq:local-multiscale}
\mathcal{L}^{-1}_{h,i}(E_{\ell,i})
:= \text{span} \left\{\mathcal{L}^{-1}_{h,i}(\psi_{\ell,i}^j): 1 \leq j \leq  n_{i,\ell}\right\},
\end{align}
where $\mathcal{L}^{-1}_{h,i} (\psi_{\ell,i}^k):=v_h\in V_{h,i}$ is the solution of the local problem
\begin{equation}
\label{eq:Li}
         a_i(v_h,w_h)=(\psi_{\ell,i}^k,w_h)_{\partial\Omega_i} \qquad\forall w_h\in V_{h,i}.
\end{equation}
%By construction, $\operatorname{dim}\mathcal{L}^{-1}_{h,i} (E_{\ell,i})=n_{i,\ell}$. 
Finally, the edge multiscale ansatz space $V_{h,0}$ is assembled from the local multiscale spaces via the weighted prolongation $\tilde{R}_i^T$:
\begin{align}\label{eq:global-multiscale}
V_{h,0} := \text{span} \left\{\tilde{R}_{i}^T\mathcal{L}^{-1}_{h,i}(\psi_{\ell,i}^j): \, 1 \leq i \leq N,\, 1 \leq j \leq  n_{i,\ell}\right\}.
\end{align}
The construction of $V_{h,0}$ is summarized in Algorithm \ref{algorithm:wavelet}.
%Let $\Phi$ be the matrix with a set of basis functions of $V_{h,0}$ being its columns. 

\begin{algorithm}[htbp]
\caption{Construction of the coarse space $V_{h,0}$.}
\label{algorithm:wavelet}
\begin{algorithmic}[1]
\REQUIRE the subdomain $\Omega_i$, the level parameter $\ell\in \mathbb{N}$
\ENSURE $V_{h,0}$
\STATE Construct the local edge space $E_{\ell,i}$ by \eqref{eq:local-edge-space}.
\STATE Compute the local multiscale space $\mathcal{L}^{-1}_{h,i} (E_{\ell,i})$ by \eqref{eq:local-multiscale}.
\STATE Construct the coarse space $V_{h,0}$ by \eqref{eq:global-multiscale}.
\end{algorithmic}
\end{algorithm}
\subsection{The REMs-HS preconditioner}
The hybrid preconditioner is a multiplicative assembly of the one-level RAS-imp preconditioner $B_{\operatorname{RAS-imp}}$ with a coarse solver determined by the coarse space $V_{h,0}\subset V_h$ \eqref{eq:global-multiscale}. Assuming the natural prolongation $R_0^T:V_{h,0} \to V_h$, we define the coarse problem $A_0 := R_0 A R_0^T : V_{h,0}\to V_{h,0}',
$ where $R_0 :V_{h}'\to V_{h,0}'$ is defined by duality. Then 
\[
\langle A_0v_h,w_h\rangle=\langle R_0AR_0^Tv_h,w_h\rangle=\langle AR_0^Tv_h,R_0^Tw_h\rangle=a(v_h,w_h),\qquad\forall v_h,w_h\in V_{h,0}.
\]
The REMs-HS preconditioner is then defined by 
\begin{equation*}
B_{\mathrm{hybrid}} : =  R_0^TA_0^{-1}R_0(I-AB_{\operatorname{RAS-imp}}) + B_{\operatorname{RAS-imp}}.
\end{equation*}
%Note that the RAS-imp preconditioner $B_{\operatorname{RAS-imp}}$ is applied only once in the computation of $B_{\mathrm{hybrid}}$: given a residual equation $Ae^{(0)}=r^{(0)}$, one first computes the one-level correction $e^{(1)}=B_{\operatorname{RAS-imp}}r^{(0)}$, then updates the residual $r^{(1)}=r^{(0)}-Ae^{(1)}$, and finally solves the coarse problem for $e^{(2)}=R_0^TA_0^{-1}R_0r^{(1)}$. The overall correction is then $B_{\mathrm{hybrid}}r^{(0)}=e^{(1)}+e^{(2)}$.
The hybridly preconditioned operator can be written as
\[
B_{\mathrm{hybrid}}A= R_0^TQ_0(I - B_{\operatorname{RAS-imp}}A) + B_{\operatorname{RAS-imp}}A,
\]
where $Q_0=A_0^{-1}R_0A:V_h\to V_{h,0}$ is the Galerkin projection onto $V_{h,0}$. Indeed,
\begin{equation}\label{eq:Q0}
a(Q_0v_h,w_h)=\langle Av_h,R_0^Tw_h\rangle=a(v_h,w_h),\qquad\forall v_h\in V_h,w_h\in V_{h,0}.
\end{equation}
Using the hybrid preconditioner $B_{\mathrm{hybrid}}$, we obtain the fixed-point iteration $u^{(n+1)}_h = u^{(n)}_h +B_{\mathrm{hybrid}} (F_h - A u^{(n)}_h)$, whose associated error transfer operator is
\[I-B_{\mathrm{hybrid}}A =(I-R_0^TQ_0)(I - B_{\operatorname{RAS-imp}}A)
 = (I-R_0^TQ_0)\biggl(I - \sum_{i=1}^N\tilde{R}^T_i Q_i\biggr).
\]
Furthermore, \eqref{eq:pou} implies
\begin{equation*}
(I-B_{\mathrm{hybrid}}A)v_h 
=(I-R_0^TQ_0)\sum_{i=1}^N \tilde{R}^T_i  \left(v_h|_{\Omega_i}- Q_i v_h\right) \qquad\forall v_h\in V_h.
\end{equation*}
Our main goal is to show that the operator $I-B_{\mathrm{hybrid}}A$ is contractive under minimal requirements on the overlap size $\delta$, the level parameter $\ell$ and under a mild resolution condition on $H$. To this end we first give an auxiliary estimate on the discrepancy between the local Schwarz projection $Q_i v_h$ and the restriction $v_h|_{\Omega_i}$.

\begin{lemma}\label{lem:decomp}
Assume $kH\ll 1$. Then, for every $v_h\in V_h$,
\begin{align}\label{eq:vhpar}
\|v_h-Q_iv_h\|_{H^1_{1/H}(\Omega_i)}\lesssim\Bigl(\frac{1}{kH}+\sqrt{\frac{H}{h}}\Bigr)\|v_h\|_{H^1_k(\Omega_i^h)},
\end{align}
where $\Omega_i^h:=\cup\{T\in\mathcal{T}_h:\overline{T}\cap\overline{\Omega}_i\neq\varnothing\}$ for $i=1,\dots,N$.
\end{lemma}
\begin{proof}
First, for any $w_h\in V_{h,i}$, a combination of the definition of $a_i(\cdot,\cdot)$ \eqref{eq:mod-sesqui} and the Schwarz inequality leads to
\begin{equation}\label{eq:lemadecomp0}\begin{aligned}
    |a_i(Q_iv_h-v_h,w_h)|&=|a(v_h,R_i^Tw_h)-a_i(v_h,w_h)|\\
    &=|(\nabla v_h,\nabla R_i^Tw_h)_{\Omega_i^h\backslash\Omega_i}-k^2(v_h,R_i^Tw_h)_{\Omega_i^h\backslash\Omega_i}-\mathrm{i}H^{-1}(v_h,w_h)_{\gamma_i}|\\
&\le\|v_h\|_{H^1_k(\Omega_i^h\backslash\Omega_i)}\|R_i^Tw_h\|_{H^1_k(\Omega_i^h\backslash\Omega_i)}+H^{-1}\|v_h\|_{L^2(\gamma_i)}\|w_h\|_{L^2(\gamma_i)}.
\end{aligned}\end{equation}
Next, we estimate the two terms on the right hand side of \eqref{eq:lemadecomp0}. For the first term in \eqref{eq:lemadecomp0}, a scaling argument gives
\[
\|R_i^Tw_h\|_{H^1_k(T)}\lesssim h^{-1/2}\|w_h\|_{L^2(\partial T\cap\textcolor{black}{\gamma_i})},\qquad\forall T\in\mathcal{T}_h \text{ with }\textcolor{black}{T\in\Omega_i^h\backslash\mathcal{T}_h(\Omega_i)}.
\]
Summing up all such $T$ over $\Omega_i^h\backslash\mathcal{T}_h(\Omega_i)$ and then using the trace inequality \eqref{eq:trace} on $\Omega_i$, we obtain
\[
\|R_i^Tw_h\|_{H^1_k(\Omega_i^h\backslash\Omega_i)}\lesssim h^{-1/2}\|w_h\|_{L^2(\gamma_i)}\lesssim\sqrt{H/h}\|w_h\|_{H^1_{1/H}(\Omega_i)}.
\]
For the second term in \eqref{eq:lemadecomp0}, the trace inequality yields
\[
\|v_h\|_{L^2(\gamma_i)}\|w_h\|_{L^2(\gamma_i)}\lesssim H\|v_h\|_{H^1_{1/H}(\Omega_i)}\|w_h\|_{H^1_{1/H}(\Omega_i)}\lesssim k^{-1}\|v_h\|_{H^1_k(\Omega_i)}\|w_h\|_{H^1_{1/H}(\Omega_i)}.
\]
Finally, choosing $w_h=Q_iv_h-v_h$ in \eqref{eq:lemadecomp0} and using the coercivity of $a_i$ \eqref{eq:coercive_modified}, we arrive at the desired assertion \eqref{eq:vhpar}.
\end{proof}

\begin{remark}
Alternatively, bounding the second term in \eqref{eq:lemadecomp0} with the trace inequality restricted to the layer $\Omega_i^h\backslash\Omega_i$ yields the cleaner estimate
\[
\|v_h-Q_iv_h\|_{H^1_{1/H}(\Omega_i)}\lesssim\sqrt{H/h}\|v_h\|_{H^1_{1/H}(\Omega_i^h\backslash\Omega_i)},\qquad\text{for }i=1,\dots,N.
\]
A similar result for the local Schwarz projection with the original impedance boundary condition was derived in \cite[Lemma 4.4]{MR4890903} with the factor $(kH)^{-3/2}(kh)^{-1/2}$, which may become large if the coarse mesh size $H$ is small.
\end{remark}

\section{Convergence analysis}\label{sec:converge}
In this section we prove that the error transfer operator $I-B_{\operatorname{hybrid}}A$ is contractive under very mild conditions on the coarse mesh size $H$ and the level parameter $\ell$.% see Theorem \ref{thm:wavelet-basedconv} and Corollary \ref{coro:main}. The proof proceeds in three steps. First, we establish \textit{a priori} estimates for functions in the local multiscale spaces $\mathcal{L}^{-1}_{h,i}(E_{\ell,i})$. Next, we derive the approximation properties of the coarse space $V_{h,0}$ to the target space $V^\partial_h$ of locally Helmholtz-harmonic discrete functions and to the solution of the adjoint problem, in Lemmas \ref{lema:last-app-harmonic} and \ref{lema:last}, respectively. Finally, we show the contractivity of $I-B_{\mathrm{hybrid}}A$ by the classical Schatz argument, combining G\aa rding's inequality \eqref{eq:garding} with the approximation property of $V_{h,0}$ to the solution of the adjoint problem.

We first introduce suitable weighted Sobolev spaces. Recall that $\|\cdot\|_{H^1_{1/H}(\Omega_i)}$ is defined in \eqref{def:local-energy}, then we introduce its weighted trace norm on $\gamma_i$ by
\[
\|v\|_{H^{1/2}_{1/H}(\gamma_i)}\coloneqq\inf_{\tilde{v}\in H^1(\Omega_i),\tilde{v}|_{\gamma_i}=v}\|\tilde{v}\|_{H^1_{1/H}(\Omega_i)},\qquad\forall v\in H^{1/2}(\gamma_i).
\]
Note that a scaling argument reveals the following equivalence relation $\|v\|_{H^{1/2}_{1/H}(\gamma_i)}\simeq |v|_{H^{1/2}(\gamma_i)}+H^{-1/2}\|v\|_{L^2(\gamma_i)}$. Let $H^{-1/2}_{1/H}(\gamma_i)\coloneqq [H^{1/2}_{1/H}(\gamma_i)]'$ be the dual space of $H^{1/2}_{1/H}(\gamma_i)$. By definition, the associated dual norm restricted to $L^2(\gamma_i)$ is
\[\|v\|_{H^{-1/2}_{1/H}(\gamma_i)}\coloneqq\sup_{w\in H^{1/2}(\gamma_i)\backslash\{0\}}\frac{(v,w)_{\gamma_i}}{\|w\|_{H^{1/2}_{1/H}(\gamma_i)}},\qquad\forall v\in L^2(\gamma_i).
\]
We also use the dual space $H^{-3/2}_*(\gamma_i)\coloneqq[H^{3/2}(\gamma_i)]'$ equiped with the norm
\[
\|v\|_{H^{-3/2}_*(\gamma_i)}\coloneqq\sup_{w\in H^2(\Omega_i)\backslash\{0\}}\frac{(v,w)_{\gamma_i}}{\|w\|_{H^2(\Omega_i)}},\qquad\forall v\in L^2(\gamma_i).
\]

\subsection{\textit{A priori} estimates at the discrete level}
%Let $f_i\in L^2(\Omega_i)$ and $g_i\in L^2(\partial\Omega_i)$, and assume that $v_h\in V_{h,i}$ is the solution of \eqref{eq:vh}. 
We first present an $L^2$-\emph{a priori} estimate for \eqref{eq:vh} involving much lower regularity assumptions on the boundary data, in sharp contrast to Lemma \ref{lema:modified}.
%Compared with Lemma \ref{lema:modified}, a weaker estimate is given in the following lemma with lower regularity requirements of the boundary data. This is the key to reducing the requirement of $\ell$.

\begin{lemma}\label{lema:vhL2}
Assume $kH\ll 1$, $f_i\in L^2(\Omega_i)$ and $g_i\in L^2(\partial\Omega_i)$. Let $v_h\in V_{h,i}$ be the solution of \eqref{eq:vh}. Then
\[
\|v_h\|_{L^2(\Omega_i)}\lesssim H^2\|f_i\|_{L^2(\Omega_i)}+H\|g_i\|_{H^{-1/2}(\Gamma_i)}+\|g_i\|_{H^{-3/2}_*(\gamma_i)}+h\|g_i\|_{H^{-1/2}_{1/H}(\gamma_i)}.
\]
\end{lemma}

\begin{proof}
First, let $z\in H^1(\Omega_i)$ be the solution of the continuous adjoint problem with right-hand side $v_h$, namely
\[
a_i(w,z)=(w,v_h)_{\Omega_i},\qquad\forall w\in H^1(\Omega_i).
\]
Then the assumption $kH\ll 1$, together with the adjoint regularity estimates \eqref{eq:uH1_modified} and \eqref{eq:uH2_modified}, implies
\begin{equation}\label{eq:zH2}
H^{-1}\|z\|_{H^1_{1/H}(\Omega_i)}+\|z\|_{H^2(\Omega_i)}\lesssim\|v_h\|_{L^2(\Omega_i)}.
\end{equation}
Next, let $z_h\in V_{h,i}$ be the discrete adjoint solution defined by
\begin{equation}\label{eq:zh}
a_i(w_h,z_h)=(w_h,v_h)_{\Omega_i},\qquad \forall w_h\in V_{h,i}.
\end{equation}
Then we obtain from the definition of $\|\cdot\|_{1/H}^{1/2}(\gamma_i)$ and Lemma \ref{lema:uerrH1_modified} that
\begin{equation}\label{eq:zerrL2}
\|z-z_h\|_{H^{1/2}_{1/H}(\gamma_i)}\le\|z-z_h\|_{H^1_{1/H}(\Omega_i)}\lesssim h\|z\|_{H^2(\Omega_i)},
\end{equation}
Now, by taking $w_h=v_h$ in \eqref{eq:zh} and $w_h=z_h$ in \eqref{eq:vh}, we derive
\begin{equation}\label{eq:lemavhL2}\begin{aligned}
\|v_h\|_{L^2(\Omega_i)}^2&=a_i(v_h,z_h)=(f_i,z_h)_{\Omega_i}+(g_i,z_h)_{\partial\Omega_i}\\
&=(f_i,z)_{\Omega_i}+(g_i,z)_{\partial\Omega_i}+(f_i,z_h-z)_{\Omega_i}+(g_i,z_h-z)_{\partial\Omega_i}.
\end{aligned}\end{equation}
We then estimate these four terms in \eqref{eq:lemavhL2}. 
For the first two terms, an application of the Schwarz inequality and the trace theorem leads to
\[
|(f_i,z)_{\Omega_i}+(g_i,z)_{\partial\Omega_i}|\le\|f_i\|_{L^2(\Omega_i)}\|z\|_{L^2(\Omega_i)}+\|g_i\|_{H^{-1/2}(\Gamma_i)}\|z\|_{H^1(\Omega_i)}+\|g_i\|_{H^{-3/2}_*(\gamma_i)}\|z\|_{H^2(\Omega_i)}.
\]
Similarly, we derive the estimate of the last two terms, 
\begin{align*}
|(f_i,z_h-z)_{\Omega_i}+(g_i,z_h-z)_{\partial\Omega_i}|&\le\|z-z_h\|_{H^1_{1/H}(\Omega_i)}
\Bigl(H\|f_i\|_{L^2(\Omega_i)}+\|g_i\|_{H^{-1/2}(\Gamma_i)}+\|g_i\|_{H^{-1/2}_{1/H}(\gamma_i)}\Bigr)\\
&\lesssim h\|z\|_{H^2(\Omega_i)}\Bigl(H\|f_i\|_{L^2(\Omega_i)}+\|g_i\|_{H^{-1/2}(\Gamma_i)}+\|g_i\|_{H^{-1/2}_{1/H}(\gamma_i)}\Bigr),
\end{align*}
where we used the finite element error estimate \eqref{eq:zerrL2} in the last inequality.

Inserting the two bounds into \eqref{eq:lemavhL2}, we obtain
\begin{align*}
\|v_h\|_{L^2(\Omega_i)}^2&\lesssim\Bigl(\|z\|_{H^1_{1/H}(\Omega_i)}+h\|z\|_{H^2(\Omega_i)}\Bigr)\Bigl(H\|f_i\|_{L^2(\Omega_i)}+\|g_i\|_{H^{-1/2}(\Gamma_i)}\Bigr)\\
&\qquad+\|z\|_{H^2(\Omega_i)}\Bigl(\|g_i\|_{H^{-3/2}_*(\gamma_i)}+h\|g_i\|_{H^{-1/2}_{1/H}(\gamma_i)}\Bigr).
\end{align*}
Finally, applying the $H^1$-adjoint and the $H^2$-adjoint regularity estimate \eqref{eq:zH2}, and dividing by $\|v_h\|_{L^2(\Omega_i)}$ yields the desired estimate.
\end{proof}
Next, we present a weighted $H^1_k$-{\em a priori} estimate for \eqref{eq:vh}.
\begin{lemma}\label{lema:vhH1}
Under assumption $\delta\le H\ll k^{-1}$, $f_i\in L^2(\Omega_i)$ and $g_i\in L^2(\partial\Omega_i)$. Let $v_h\in V_{h,i}$ be the solution of \eqref{eq:vh}. Then
\[
\delta\|\chi_iv_h\|_{H^1_k(\Omega_i)}\lesssim H^2\|f_i\|_{L^2(\Omega_i)}+H\|g_i\|_{H^{-1/2}(\Gamma_i)}+\|g_i\|_{H^{-3/2}_*(\gamma_i)}+h\|g_i\|_{H^{-1/2}_{1/H}(\gamma_i)}.
\]
\end{lemma}

\begin{proof}
Choose $w_h=\Pi_h(\chi_i^2v_h)\in V_{h,i}^0$ in \eqref{eq:vh}, then
\begin{align*}
(f_i&,\Pi_h(\chi_i^2v_h))_{\Omega_i}+(g_i,\Pi_h(\chi_i^2v_h))_{\Gamma_i}=a_i(v_h,\Pi_h(\chi_i^2v_h))=(\nabla v_h,\nabla(\Pi_h-I)(\chi_i^2v_h))_{\Omega_i}\\
&\qquad+(\nabla v_h,\chi_i^2\nabla v_h+2\chi_iv_h\nabla\chi_i)_{\Omega_i}-k^2(v_h,\Pi_h(\chi_i^2v_h))_{\Omega_i}+\mathrm{i}k(v_h,\Pi_h(\chi_i^2v_h))_{\Gamma_i}.
\end{align*}
Taking the real part of both sides, we obtain
\begin{equation}\label{eq:lemavhH10}
\begin{aligned}
\|\chi_i\nabla v_h\|_{L^2(\Omega_i)}^2&=\Re\Bigl((\nabla v_h,\nabla(I-\Pi_h)(\chi_i^2v_h))_{\Omega_i}-2(\chi_i\nabla v_h,v_h\nabla\chi_i)_{\Omega_i}\\
&\qquad+(f_i+k^2v_h,\Pi_h(\chi_i^2v_h))_{\Omega_i}+(g_i-\mathrm{i}kv_h,\Pi_h(\chi_i^2v_h))_{\Gamma_i}\Bigr).
\end{aligned}
\end{equation}
Next, we estimate the four terms on the right hand side of \eqref{eq:lemavhH10}.
An application of the inverse inequality for $v_h$ and the approximation property \eqref{approxi} of $\Pi_h$ gives the estimate of the first two terms,
\begin{align*}
&\quad|(\nabla v_h,\nabla(I-\Pi_h)(\chi_i^2v_h))_{\Omega_i}-2(\chi_i\nabla v_h,v_h\nabla\chi_i)_{\Omega_i}|\\&\lesssim\|v_h\|_{L^2(\Omega_i)}\|\nabla^2_h(\chi_i^2v_h)\|_{L^2(\Omega_i)}+\|\chi_i\nabla v_h\|_{L^2(\Omega_i)}\|v_h\nabla\chi_i\|_{L^2(\Omega_i)}\\
&\lesssim\|v_h\|_{L^2(\Omega_i)}\|\nabla\chi_i\|_{L^\infty(\Omega_i)}\Bigl(\|\chi_i\nabla v_h\|_{L^2(\Omega_i)}+\|\nabla\chi_i\|_{L^\infty(\Omega_i)}\|v_h\|_{L^2(\Omega_i)}\Bigr)\\
&\le\frac{1}{4}\|\chi_i\nabla v_h\|_{L^2(\Omega_i)}^2+\frac{C}{\delta^2}\|v_h\|_{L^2(\Omega_i)}^2,
\end{align*}
where the second-order derivatives for $\chi_i$ and $v_h$ vanish since $\chi_i,v_h\in V_{h,i}$ and $\nabla^2_h$ denotes the piecewise Hessian operator over $\mathcal{T}_h$.

To estimate the last two terms in \eqref{eq:lemavhH10}, we first apply the stability \eqref{eq:Pihstability} of $\Pi_h$ and the dualities to derive
\begin{align*}
&\quad|(f_i+k^2v_h,\Pi_h(\chi_i^2v_h))_{\Omega_i}+(g_i-\mathrm{i}kv_h,\Pi_h(\chi_i^2v_h))_{\Gamma_i}|\\&\lesssim \|\chi_i^2v_h\|_{H^1_{1/H}(\Omega_i)}\times\Bigl(H\|f_i+k^2v\|_{L^2(\Omega_i)}+\|g_i-\mathrm{i}kv_h\|_{H^{-1/2}(\Gamma_i)}\Bigr)\\
&\lesssim\Bigl(\|\chi_i\nabla v_h\|_{L^2(\Omega_i)}+\delta^{-1}\|v_h\|_{L^2(\Omega_i)}\Bigr)\Bigl(H\|f_i\|_{L^2(\Omega_i)}+\|g_i\|_{H^{-1/2}(\Gamma_i)}+k\|v_h\|_{L^2(\Omega_i)}\Bigr)\\
&\le\frac{1}{4}\|\chi_i\nabla v_h\|_{L^2(\Omega_i)}^2+C\Bigl((\delta^{-2}+k^2)\|v_h\|_{L^2(\Omega_i)}^2+H^2\|f_i\|_{L^2(\Omega_i)}^2+\|g_i\|_{H^{-1/2}(\Gamma_i)}^2\Bigr).
\end{align*}

Substituting the above estimates into \eqref{eq:lemavhH10}, we obtain the discrete Caccioppoli-type inequality 
\[
\|\chi_i\nabla v_h\|_{L^2(\Omega_i)}\lesssim(\delta^{-1}+k)\|v_h\|_{L^2(\Omega_i)}+H\|f_i\|_{L^2(\Omega_i)}+\|g_i\|_{H^{-1/2}(\Gamma_i)}.
\]
Therefore, by definition,
\begin{align*}
\|\chi_iv_h\|_{H^1_k(\Omega_i)}&\le\|\chi_i\nabla v_h\|_{L^2(\Omega_i)}+\|v_h\nabla\chi_i\|_{L^2(\Omega_i)}+k\|\chi_iv_h\|_{L^2(\Omega_i)}\\
&\lesssim(\delta^{-1}+k)\|v_h\|_{L^2(\Omega_i)}+H\|f_i\|_{L^2(\Omega_i)}+\|g_i\|_{H^{-1/2}(\Gamma_i)}.
\end{align*}
Substituting the $L^2$-bound for $v_h$ from Lemma \ref{lema:vhL2} into this inequality yields the desired estimate under the assumption $\delta\le H\ll k^{-1}$.
\end{proof}

Conversely, for every $v_h\in V_{h,i}^\partial$, one can determine a boundary datum $g_h\in V_{h,i}|_{\gamma_i}$ such that
\begin{equation}\label{eq:gh}
    (g_h,w_h)_{\gamma_i}=a_i(v_h,w_h),\qquad \forall w_h\in V_{h,i}|_{\gamma_i},
\end{equation}
see \cite[Proposition 4.1]{MR4890903}. Its well-posedness is established with the aid of a finite element lifting and of the projection onto $V_{h,i}|_{\gamma_i}$.

\begin{lemma}\label{lema:lift}
For any $v\in H^{1/2}(\gamma_i)$, there exists a lifting $E_{h,i}v\in V_{h,i}$ such that 
\[
\|E_{h,i}v\|_{H^1_{1/H}(\Omega_i)}+h^{-1/2}\|v-E_{h,i}v\|_{L^2(\gamma_i)}\lesssim\|v\|_{H^{1/2}_{1/H}(\gamma_i)}.
\]
Moreover, if $v\in V_{h,i}|_{\gamma_i}$, then $E_{h,i}v=v$ on $\gamma_i$.
\end{lemma}

\begin{proof}
Let $I_{h,i}:H^1(\Omega_i)\to V_{h,i}$ be the Scott--Zhang interpolation operator \cite{ScottZhang:1990},
\[
\|\nabla I_{h,i}w\|_{L^2(\Omega_i)}+h^{-1}\|w-I_{h,i}w\|_{L^2(\Omega_i)}\lesssim\|\nabla w\|_{L^2(\Omega_i)},\qquad\forall w\in H^1(\Omega_i),
\]
and $I_{h,i}$ preserves the trace whenever $w|_{\gamma_i}\in V_{h,i}|_{\gamma_i}$.

By definition of the weighted trace norm, there exists $\tilde{v}\in H^1(\Omega_i)$ with $\tilde{v}|_{\gamma_i}=v$ and $\|\tilde{v}\|_{H^1_{1/H}(\Omega_i)}\le2\|v\|_{H^{1/2}_{1/H}(\gamma_i)}$. Setting $E_{h,i}v:=I_{h,i}\tilde{v}$ and using $h\le H$, we obtain
\[
\|E_{h,i}v\|_{H^1_{1/H}(\Omega_i)}\le\|\tilde{v}\|_{H^1_{1/H}(\Omega_i)}+\|\tilde{v}-I_{h,i}\tilde{v}\|_{H^1_{1/H}(\Omega_i)}\lesssim\|\tilde{v}\|_{H^1_{1/H}(\Omega_i)}\lesssim\|v\|_{H^{1/2}_{1/H}(\gamma_i)},
\]
and, by the trace inequality \eqref{eq:trace},
\[
\|v-E_{h,i}v\|_{L^2(\gamma_i)}\lesssim h^{-1/2}\|\tilde{v}-I_{h,i}\tilde{v}\|_{L^2(\Omega_i)}+\sqrt{h}\|\nabla(\tilde{v}-I_{h,i}\tilde{v})\|_{L^2(\Omega_i)}\lesssim\sqrt{h}\|v\|_{H^{1/2}_{1/H}(\gamma_i)}.
\]
Combining the last two inequalities proves the first assertion. Moreover, the trace-preserving property of $I_{h,i}$ gives $E_{h,i}v=\tilde{v}=v$ on $\gamma_i$ whenever $v\in V_{h,i}|_{\gamma_i}$.
\end{proof}

Let $P_{h,i}:L^2(\gamma_i)\to V_{h,i}|_{\gamma_i}$ be the $L^2(\gamma_i)$-orthogonal projection,
\[
(P_{h,i}v,w_h)_{\gamma_i}=(v,w_h)_{\gamma_i},\qquad \forall w_h\in V_{h,i}|_{\gamma_i},
\]
and let $\overline{P}_{h,i}:L^2(\gamma_i)\to V_{h,i}|_{\gamma_i}$ be defined by duality,
\[
(w_h,\overline{P}_{h,i}v)_{\gamma_i}=(w_h,v)_{\gamma_i},\qquad \forall w_h\in V_{h,i}|_{\gamma_i}.
\]
Then $P_{h,i}$ and $\overline{P}_{h,i}$ have analogous stability and approximation properties.
\begin{lemma}\label{lema:Phi}
The following two statements hold. 
\begin{enumerate}
\item $P_{h,i}$ is stable in the weighted norms in the sense that
\[
\|P_{h,i}v\|_{H^s_{1/H}(\gamma_i)}\lesssim\|v\|_{H^{s}_{1/H}(\gamma_i)},\qquad \forall v\in H^s(\gamma_i)\cap L^2(\gamma_i)\text{ with } s=\pm1/2.
\]
\item The negative-norm approximation estimate
\begin{equation}\label{eq:H-32err}
\|v-P_{h,i}v\|_{H^{-3/2}_*(\gamma_i)}\lesssim h\|v\|_{H^{-1/2}_{1/H}(\gamma_i)}\qquad\forall v\in L^2(\gamma_i).
\end{equation}
\end{enumerate}
\end{lemma}

\begin{proof}
The operator $P_{h,i}$ is stable in $L^2(\gamma_i)$ that $\|P_{h,i}v\|_{L^2(\gamma_i)}\le\|v\|_{L^2(\gamma_i)}$.
If $v\in H^{1/2}(\gamma_i)$, write $P_{h,i}v=E_{h,i}v+P_{h,i}(v-E_{h,i}v)$ on $\gamma_i$, where $E_{h,i}v\in V_{h,i}$ is the lifting from Lemma \ref{lema:lift}. Then
\[
\|P_{h,i}v\|_{H^{1/2}_{1/H}(\gamma_i)}\lesssim\|E_{h,i}v\|_{H^{1/2}_{1/H}(\gamma_i)}+h^{-1/2}\|v-E_{h,i}v\|_{L^2(\gamma_i)}\lesssim\|v\|_{H^{1/2}_{1/H}(\gamma_i)},
\]
where the second step uses the inverse inequality and the $L^2(\gamma_i)$-stability of $P_{h,i}$,
\[
\|P_{h,i}(v-E_{h,i}v)\|_{H^{1/2}_{1/H}(\gamma_i)}\lesssim h^{-1/2}\|P_{h,i}(v-E_{h,i}v)\|_{L^2(\gamma_i)}\lesssim h^{-1/2}\|v-E_{h,i}v\|_{L^2(\gamma_i)}.
\]
For the $H^{-1/2}_{1/H}(\gamma_i)$-stability, by the definitions of $\overline{P}_{h,i}$ and $P_{h,i}$,
\[
 \left|(P_{h,i}v,w)_{\gamma_i}\right|=\left|(v,\overline{P}_{h,i}w)_{\gamma_i}\right|\le\|v\|_{H^{-1/2}_{1/H}(\gamma_i)}\|w\|_{H^{1/2}_{1/H}(\gamma_i)},\qquad\forall w\in H^{1/2}(\gamma_i),
\]
where the last inequality uses the $H^{1/2}_{1/H}(\gamma_i)$-stability of $\overline{P}_{h,i}$. Dividing by $\|w\|_{H^{1/2}_{1/H}(\gamma_i)}$ and taking the supremum over $w$ yields the $H^{-1/2}_{1/H}(\gamma_i)$-stability of $P_{h,i}$.

Finally, we prove the negative-norm approximation \eqref{eq:H-32err}. For any $w\in H^2(\Omega_i)$, the definition of $P_{h,i}$ implies,
\begin{align*}
|(v-P_{h,i}v,w)_{\gamma_i}|&=|(v-P_{h,i}v,w-\Pi_hw)_{\gamma_i}\le\|v-P_{h,i}v\|_{H^{-1/2}_{1/H}(\gamma_i)}\|w-\Pi_hw\|_{H^{1/2}_{1/H}(\gamma_i)}\\
&\lesssim \|v\|_{H^{-1/2}_{1/H}(\gamma_i)}\Bigl(\|\nabla(w-\Pi_hw)\|_{L^2(\Omega_i)}+H^{-1}\|w-\Pi_hw\|_{L^2(\Omega_i)}\Bigr)\\
&\lesssim h\|v\|_{H^{-1/2}_{1/H}(\gamma_i)}\|w\|_{H^2(\Omega_i)},
\end{align*}
where the second inequality uses the $H^{-1/2}_{1/H}(\gamma_i)$-stability of $P_{h,i}$ and the definition of the $H^{1/2}_{1/H}(\gamma_i)$-norm, and the last inequality uses the interpolation error estimate \eqref{approxi}. Dividing by $\|w\|_{H^2(\Omega_i)}$ and taking the supremum over $w$ gives \eqref{eq:H-32err}, and this completes the proof.
\end{proof}
Now we are ready to present the boundedness of the impedance map \eqref{eq:gh}.
\begin{lemma}[Boundedness of impedance maps]\label{lema:ghH-12}
Assume $kH\ll 1$, and let $g_h\in V_{h,i}|_{\gamma_i}$ be the solution of \eqref{eq:gh}. Then
\[
\|g_h\|_{H^{-1/2}_{1/H}(\gamma_i)}\lesssim\|v_h\|_{H^1_{1/H}(\Omega_i)}.
\]
\end{lemma}

\begin{proof}
For any $w\in H^{1/2}(\gamma_i)$, choose the finite element lifting $E_{h,i}\overline{P}_{h,i}w\in V_{h,i}$ in Lemma \ref{lema:lift} such that $E_{h,i}\overline{P}_{h,i}w|_{\gamma_i}=\overline{P}_{h,i}w$ and
\[
\|E_{h,i}\overline{P}_{h,i}w\|_{H^1_{1/H}(\Omega_i)}\lesssim\|\overline{P}_{h,i}w\|_{H^{1/2}_{1/H}(\gamma_i)}\lesssim\|w\|_{H^{1/2}_{1/H}(\gamma_i)},
\]
where we use the $H^{1/2}_{1/H}(\gamma_i)$-stability of $\overline{P}_{h,i}$ in the last step.

Using the discrete variational identity \eqref{eq:vh} with this particular test function $w_h=E_{h,i}\overline{P}_{h,i}w$, we obtain
\begin{align*}
|(g_h,w)_{\gamma_i}|=|(g_h,\overline{P}_{h,i}w)_{\gamma_i}|=|a_i(v_h,E_{h,i}\overline{P}_{h,i}w)|\lesssim \|v_h\|_{H^1_{1/H}(\Omega_i)}\|E_{h,i}\overline{P}_{h,i}w\|_{H^1_{1/H}(\Omega_i)},
\end{align*}
where the boundedness of $a_i(\cdot,\cdot)$ \eqref{eq:bound_modified} was used in the last inequality. 

Finally, combining the above two inequalities and taking the supremum over all $w\in H^{1/2}(\gamma_i)$ yields the desired estimate.
\end{proof}

\subsection{Approximation and contractivity}
We now derive approximation properties of the edge multiscale ansatz space $V_{h,0}$ to the target space $V^\partial_h$ of locally Helmholtz-harmonic discrete functions, as well as to the solution of the adjoint problem. %As in \cite[Theorem 4.2]{MR4860976}, the proof is based on the global and local splittings, but uses the weighted prolongation $\tilde{R}^T_i$ in place of the partition-of-unity functions $\chi_i$. 
\begin{lemma}[Approximation properties of the edge multiscale ansatz space $V_{h,0}$ to $V^\partial_h$] \label{lema:last-app-harmonic}
Assume $\delta\le H\ll k^{-1}$ and $h\le 2^{-\ell}H$. For any $v_h=\sum_{i=1}^N\tilde{R}_i^Tv_i\in V^\partial_h$ as in \eqref{eq:glo-edge-space}, with $v_i\in V_{h,i}^\partial$ as in \eqref{eq:local-space}, there holds
\begin{equation}\label{eq:waveletErrconv1}
\inf_{v_0\in V_{h,0}}\|v_h-v_0\|_{H^1_k(\Omega)}\lesssim2^{-\ell}\frac{H}{\delta}\biggl(\sum_{i=1}^N\| v_{i}\|_{H^1_{1/H}(\Omega_i)}^2\biggr)^{1/2}.
\end{equation}
\end{lemma}

\begin{proof}
As shown in \cite[Proposition 4.1]{MR4890903}, for each $v_i\in V_{h,i}^\partial$ there exists $g_i\in V_{h,i}|_{\gamma_i}$ such that
\[
a_i(v_i,w_h)=(g_i,w_h)_{\gamma_i},\qquad\forall w_h\in V_{h,i}.
\]
Therefore, by Lemma \ref{lema:ghH-12}, under the assumption $h\le H\ll k^{-1}$, the solution $g_i$ is well-posed and $\|g_i\|_{H^{-1/2}_{1/H}(\gamma_i)}\lesssim\|v_i\|_{H^1_{1/H}(\Omega_i)}$.

By the definition of $\mathcal{L}_{h,i}^{-1}$ in \eqref{eq:Li},
\[
a_i(v_i-\mathcal{L}_{h,i}^{-1}\mathcal{P}_{\ell,i}g_{i}, w_h)=(g_i-\mathcal{P}_{\ell,i}g_i,w_h)_{\gamma_i},\qquad\forall w_h\in V_{h,i}.
\]
Since $v_i-\mathcal{L}_{h,i}^{-1}\mathcal{P}_{\ell,i}g_{i}\in V_{h,i}$ satisfies \eqref{eq:vh} with impedance boundary datum $g_i-\mathcal{P}_{\ell,i}g_i\in V_{h,i}|_{\gamma_i}$, Lemma \ref{lema:vhH1} gives 
\begin{align*}
\delta\|\chi_i(v_i-\mathcal{L}_{h,i}^{-1}\mathcal{P}_{\ell,i}g_{i})\|_{H^1_k(\Omega_i)}&\lesssim\|g_i-\mathcal{P}_{\ell,i}g_i\|_{H^{-3/2}_*(\gamma_i)}+h\|g_i-\mathcal{P}_{\ell,i}g_i\|_{H^{-1/2}_{1/H}(\gamma_i)}\\
&\lesssim (2^{-\ell}H+h)\|g_i\|_{H^{-1/2}_{1/H}(\gamma_i)}\lesssim 2^{-\ell}H\|v_i\|_{H^1_{1/H}(\Omega_i)},
\end{align*}
where the second line uses the $H^{-3/2}_*(\gamma_i)$-approximation error and the $H^{-1/2}_{1/H}(\gamma_i)$-stability of $\mathcal{P}_{\ell,i}$, analogous to Lemma \ref{lema:Phi}.

Finally, for the approximation $v_0=\sum_{i=1}^N\tilde{R}_i^T\mathcal{L}_{h,i}^{-1}\mathcal{P}_{\ell,i}g_{i}\in V_{h,0}$, Lemma \ref{lema:pum-2} yields
\[
\|v_h-v_0\|_{H^1_k(\Omega)}^2\lesssim\sum_{i=1}^N\|\chi_i(v_i-\mathcal{L}_{h,i}^{-1}\mathcal{P}_{\ell,i}g_{i})\|_{H^1_k(\Omega_i)}^2\lesssim \Bigl(2^{-\ell}\frac{H}{\delta}\Bigr)^2\sum_{i=1}^N\|v_i\|_{H^1_{1/H}(\Omega_i)}^2.
\]
This proves \eqref{eq:waveletErrconv1}. 
\end{proof}

\begin{lemma}[Approximation properties of the coarse space $V_{h,0}$ to the adjoint problem] \label{lema:last}
Assume $\delta\le H\ll k^{-1}$ and $h\le 2^{-\ell}H$. For any $f\in L^2(\Omega)$, let $v\in H^1(\Omega)$ satisfy
\[
a(w,v)=(w,f)_\Omega,\qquad\forall w\in H^1(\Omega).
\]
Then
\begin{equation}\label{eq:waveletErrconvDual}
\inf_{v_0\in V_{h,0}}\|v-v_0\|_{H^1_k(\Omega)}\lesssim\Bigl(kh+\frac{H}{\delta}(H+2^{-2\ell})\Bigr)\|f\|_{L^2(\Omega)}.
\end{equation}
\end{lemma}
\begin{proof}
For each $\Omega_i$, let $v_i\in V_{h,i}$ be the solution of
\[
a_i(v_i,w_h)=(f,w_h)_{\Omega_i}+(g_i,w_h)_{\partial\Omega_i},\qquad\forall w_h\in V_{h,i},
\]
where $g_i\in L^2(\partial\Omega_i)$ is given by $g_i=2\mathrm{i}kv$ on $\Gamma_i$ and $g_i=\partial_nv+\mathrm{i}H^{-1}v$ on $\gamma_i$. Clearly, $v_i$ is the finite element approximation of $v|_{\Omega_i}$. Hence, Lemma \ref{lema:uerrH1_modified} implies
\begin{equation}\label{eq:lemalast0}
    \|\chi_i(v-v_i)\|_{H^1_k(\Omega_i)}\lesssim\|v-v_i\|_{H^1_k(\Omega_i)}+\delta^{-1}\|v-v_i\|_{L^2(\Omega_i)}\lesssim h\|v\|_{H^2(\Omega_i)}.
\end{equation}
By the definition of $\mathcal{L}_{h,i}^{-1}$ \eqref{eq:Li},
\[
a_i(v_i-\mathcal{L}_{h,i}^{-1}\mathcal{P}_{\ell,i}g_i,w_h)=(f,w_h)_{\Omega_i}+(g_i-\mathcal{P}_{\ell,i}g_i,w_h)_{\partial\Omega_i},\qquad\forall w_h\in V_{h,i}.
\]
Since $v_i-\mathcal{L}_{h,i}^{-1}\mathcal{P}_{\ell,i}g_{i}\in V_{h,i}$ satisfies \eqref{eq:vh} with source datum $f\in L^2(\Omega_i)$ and impedance boundary datum $g_i-\mathcal{P}_{\ell,i}g_i\in L^2(\partial\Omega_i)$, Lemma \ref{lema:vhH1} gives 
\begin{equation}\label{eq:lemalast1}
\begin{aligned}
\delta\|\chi_i(v_i-\mathcal{L}_{h,i}^{-1}\mathcal{P}_{\ell,i}g_{i}\|_{H^1_k(\Omega_i)}&\lesssim H^2\|f\|_{L^2(\Omega_i)}+H\|g_i-\mathcal{P}_{\ell,i}g_i\|_{H^{-1/2}(\Gamma_i)}\\
&\qquad+\|g_i-\mathcal{P}_{\ell,i}g_i\|_{H^{-3/2}_*(\gamma_i)}+h\|g_i-\mathcal{P}_{\ell,i}g_i\|_{H^{-1/2}_{1/H}(\gamma_i)}\\
&\lesssim H^2\Bigl(\|f\|_{L^2(\Omega_i)}+2^{-2\ell}H\|g_i\|_{H^{3/2}(\Gamma_i)}+2^{-2\ell}\|g_i\|_{H^{1/2}(\gamma_i)}\Bigr),
\end{aligned}\end{equation}
where the last line follows from the approximation properties \eqref{eq:approx-wavelets}. To illustrate the last estimate, consider the second term: for all $w\in H^{1/2}(\Gamma_i)$,
\[
(g_i-\mathcal{P}_{\ell,i}g_i,w)_{\Gamma_i}=(g_i-\mathcal{P}_{\ell,i}g_i,w-\mathcal{P}_{\ell,i}w)_{\Gamma_i}\lesssim 2^{-2\ell}H^2\|g_i\|_{H^{3/2}(\Gamma_i)}\|w\|_{H^{1/2}(\Gamma_i)},
\]
which shows that $\|g_i-\mathcal{P}_{\ell,i}g_i\|_{H^{-1/2}(\Gamma_i)}\lesssim 2^{-2\ell}H^2\|g_i\|_{H^{3/2}(\Gamma_i)}$.

Then an application of the triangle inequality, together with \eqref{eq:lemalast0}, \eqref{eq:lemalast1} and the trace inequality, yields
\begin{align*}
    \|\chi_i(v-\mathcal{L}_{h,i}^{-1}\mathcal{P}_{\ell,i}g_{i}\|_{H^1_k(\Omega_i)}\lesssim\frac{H^2}{\delta}\Bigl(\|f\|_{L^2(\Omega_i)}+2^{-2\ell}\|v\|_{H^2(\Omega_i)}+2^{-2\ell}H^{-1}\|v\|_{H^1(\Omega_i)}\Bigr)+h\|v\|_{H^2(\Omega_i)}.
\end{align*}
Denote $v_0\coloneqq\sum_{i=1}^N\tilde{R}_i^T\mathcal{L}_{h,i}^{-1}\mathcal{P}_{\ell,i}g_{i}\in V_{h,0}$, then we derive from Lemma \ref{lema:pum-2} combined with \eqref{eq:sumpou-ieq}, 
\begin{align*}
\|v-v_0\|_{H^1_k(\Omega)}&\lesssim\Bigl(\sum_{i=1}^N\|\chi_i(v_i-\mathcal{L}_{h,i}^{-1}\mathcal{P}_{\ell,i}g_{i})\|_{H^1_k(\Omega_i)}^2\Bigr)^{1/2}\\
&\lesssim h\|v\|_{H^2(\Omega)}+\frac{H^2}{\delta}\Bigl(\|f\|_{L^2(\Omega)}+2^{-2\ell}\|v\|_{H^2(\Omega)}+2^{-2\ell}H^{-1}\|v\|_{H^1(\Omega)}\Bigr)\\
&\lesssim \Bigl(kh+\frac{H}{\delta}(H+2^{-2\ell})\Bigr)\|f\|_{L^2(\Omega)},
\end{align*}
where the regularities \eqref{eq:uH1} and \eqref{eq:uH2} of the adjoint Helmholtz equation were used in the last line. This proves \eqref{eq:waveletErrconvDual}. 
\end{proof}
Next, we derive the well-posedness of $A_0$.
\begin{lemma}[Inf-sup condition for $A_0$]\label{lem:infsup}
Assume $k^2h\ll1$, $h\le 2^{-\ell}H$ and $kH(H+2^{-2\ell})\ll\delta\le H$. Then
\[
\sup_{w_h\in V_{h,0}}\frac{\Re a(v_h,w_h)}{\|w_h\|_{H^1_k(\Omega)}}\gtrsim\frac{1}{k}\|v_h\|_{H^1_k(\Omega)},\qquad\forall v_h\in V_{h,0}.
\]
\end{lemma}

\begin{proof}
We establish the inf-sup condition by an Aubin--Nitsche argument. For fixed $v_h\in V_{h,0}$, we seek $z\in H^2(\Omega)$ satisfying
\begin{align*}
a(w,z)=(w,v_h)\qquad\forall v_h\in V_{h,0}.
\end{align*}
%\[\left\{\begin{aligned}
%-\Delta z-k^2z&=v_h&&\text{in }\Omega, \\
%\partial_nz-\mathrm{i}kz&=0&&\text{on }\partial\Omega.
%\end{aligned}\right.
%\]
Then Lemma \ref{lema:last} implies the existence of $z_0\in V_{h,0}$, s.t., 
\[
\|z-z_0\|_{H^1_k(\Omega)}\lesssim\Bigl(kh+\frac{H}{\delta}(H+2^{-2\ell})\Bigr)\|v_h\|_{L^2(\Omega)}.
\]
Consider the test function $w_h\coloneqq v_h+2k^2z_0\in V_{h,0}$. On the one hand, a direct computation using the definition of $z\in H^2(\Omega)$ leads to
\begin{align*}
a(v_h,w_h)=a(v_h,v_h)+2k^2a(v_h,z)+2k^2a(v_h,z_0-z)=\|v_h\|_{H^1_k(\Omega)}^2+\mathrm{i}{k}\|v_h\|_{L^2(\partial\Omega)}^2+2k^2a(v_h,z_0-z).
\end{align*}
Consequently, taking the absolute value over the real part yields
\begin{align*}
\Re a(v_h,w_h)&\ge\|v_h\|_{H^1_k(\Omega)}^2-2k^2|a(v_h,z-z_0)|\gtrsim\|v_h\|_{H^1_k(\Omega)}^2-k^2\Bigl(kh+\frac{H}{\delta}(H+2^{-2\ell})\Bigr)\|v_h\|_{H^1_k(\Omega)}\|v_h\|_{L^2(\Omega)}\\
&\gtrsim\Bigl(1-k^2h-\frac{kH}{\delta}(H+2^{-2\ell})\Bigr)\|v_h\|_{H^1_k(\Omega)}^2.
\end{align*}
On the other hand, the norm of $w_h$ can be bounded by the triangle inequality,
\begin{align*}
\|w_h\|_{H^1_k(\Omega)}&\le\|v_h\|_{H^1_k(\Omega)}+2k^2\|z\|_{H^1_k(\Omega)}+2k^2\|z-z_0\|_{H^1_k(\Omega)}\\
&\lesssim\|v_h\|_{H^1_k(\Omega)}+k^2\Bigl(1+kh+\frac{H}{\delta}(H+2^{-2\ell})\Bigr)\|v_h\|_{L^2(\Omega)}\\
&\lesssim k\Bigl(1+kh+\frac{H}{\delta}(H+2^{-2\ell})\Bigr)\|v_h\|_{H^1_k(\Omega)},
\end{align*}
where the adjoint $H^1$-regularity \eqref{eq:uH1} for $z$ was used in the second line. Combining the two inequalities yields the desired inf-sup estimate under the stated assumptions.
\end{proof}
\begin{remark}\label{rmk:imp}
Lemma \ref{lema:last} explains why mIBC is needed in the construction of the coarse space $V_{h,0}$. With the original impedance boundary condition, the coefficient multiplying $\|f\|_{L^2(\Omega_i)}$ in \eqref{eq:lemalast1} would be only of $\mathcal{O}(H/k)$, which is insufficient to obtain the inf-sup condition in Lemma \ref{lem:infsup} since it means we need $k/\delta\cdot\mathcal{O}(H/k)\ll 1$. Alternatively, one would need to enrich the coarse space with additional volume-source basis functions functions to obtain the same order.
\end{remark}
Finally, we derive the contractivity of the operator $I-B_{\operatorname{hybrid}}A$.
\begin{theorem}\label{thm:wavelet-basedconv}
Assume that $k^2h\ll1$, $h\lesssim\min(2^{-\ell}H,k^2H^3)$ and $kH(H+2^{-2\ell})\ll\delta\le H$. Then
\[
\|(I-B_{\mathrm{hybrid}}A)v_h\|_{H^1_k(\Omega)}\lesssim\frac{2^{-\ell}H}{\delta}\sqrt{\frac{H}{h}}\|v_h\|_{H^1_k(\Omega)},\qquad\forall v_h\in V_h.
\]
\end{theorem}
\begin{proof}
Let $e_h:=(I-B_{\mathrm{hybrid}}A)v_h=(I-R_0^TQ_0)v_h^{\partial}$ with $v_h^\partial=\sum_{i=1}^N\tilde{R}_i^T(v_h|_{\Omega_i}-Q_iv_h)\in V_h^\partial$. Since $V_{h,0}\subset V_h$, the Galerkin orthogonality \eqref{eq:Q0} implies
\begin{equation}\label{eq:ehorth}
a(e_h,w_h)=a(v_h^\partial-Q_0v_h^\partial,w_h)=0,\qquad\forall w_h\in V_{h,0}.
\end{equation}

First, we estimate $\|e_h\|_{L^2(\Omega)}$ by an Aubin--Nitsche argument. Let $z\in H^1(\Omega)$ be the solution of the adjoint problem
\begin{equation}\label{eq:prop0}
a(w,z)=(w,e_h)_\Omega,\qquad\forall w\in H^1(\Omega).
\end{equation}
By Lemma \ref{lema:last}, there exists $z_0\in V_{h,0}$ with
\[
\|z-z_0\|_{H^1_k(\Omega)}\lesssim\Bigl(kh+\frac{H}{\delta}(H+2^{-2\ell})\Bigr)\|e_h\|_{L^2(\Omega)}.
\]
Choosing $w=e_h$ in \eqref{eq:prop0} and $w_h=z_0$ in \eqref{eq:ehorth}, we obtain
\begin{align*}
\|e_h\|_{L^2(\Omega)}^2&=a(e_h,z)=a(e_h,z-z_0)\lesssim\|e_h\|_{H^1_k(\Omega)}\|z-z_0\|_{H^1_k(\Omega)}\lesssim\Bigl(kh+\frac{H}{\delta}(H+2^{-2\ell})\Big)\|e_h\|_{H^1_k(\Omega)}\|e_h\|_{L^2(\Omega)}.
\end{align*}
This implies
\[
\|e_h\|_{L^2(\Omega)}
\lesssim\Bigl(kh+\frac{H}{\delta}(H+2^{-2\ell})\Big)\|e_h\|_{H^1_k(\Omega)}.
\]

G\aa rding's inequality in Lemma \ref{lema:sesquilinear} and the orthogonality \eqref{eq:ehorth} imply
\[
\|e_h\|_{H^1_k(\Omega)}^2=\Re a(e_h,v_h^{\partial}-v_{0})
+2k^2\|e_h\|_{L^2(\Omega)}^2, \qquad\forall v_0\in V_{h,0}.
\]
Absorbing the last term on the right with the above $L^2$-estimate using $k^2h\ll1$ and $kH(H+2^{-2\ell})\ll\delta$, we get
\begin{align*}
\|e_h\|_{H^1_k(\Omega)}\lesssim \Re a(e_h,v_h^{\partial}-v_0)\qquad\forall v_0\in V_{h,0}.
\end{align*}
Finally, the boundedness of the sesquilinear form $a(\cdot,\cdot)$ yields
\begin{align*}
\|e_h\|_{H^1_k(\Omega)}^2&\lesssim \inf_{v_0\in V_{h,0}}\|v_h^{\partial}-v_0\|_{H^1_k(\Omega)}^2\lesssim\Bigl(2^{-\ell}\frac{H}{\delta}\Bigr)^2\sum_{i=1}^N\|v_h-Q_iv_h\|_{H^1_{1/H}(\Omega_i)}^2\lesssim\frac{2^{-2\ell}H^3}{\delta^2h}\|v_h\|_{H^1_k(\Omega)}^2,
\end{align*}
where the approximation property \eqref{eq:waveletErrconv1} with $v_h-Q_iv_h\in V_{h,i}^\partial$ and the estimate \eqref{eq:vhpar} were used in the last two steps.
\end{proof}
A direct application of Theorem \ref{thm:wavelet-basedconv} reveals the following result.
\begin{corollary}\label{coro:main}
Assume $h\simeq k^{-2-\alpha}$, $H\simeq k^{-1-\beta}$, $\delta\simeq k^{-1-\gamma}$, $2^{-\ell}\simeq k^{-1/2-\zeta}$ with small real numbers $\alpha,\beta,\gamma,\zeta\in(0,1/4]$ satisfying $2\beta>\gamma\ge\beta$, $\beta+2\zeta>\gamma$ and $3\beta+2\zeta>2\gamma+\alpha$. Then $\|I-B_{\operatorname{hybrid}}A\|<1$ for sufficiently large $k\ge k_0$.
\end{corollary}
\begin{proof}
The assumptions of Theorem \ref{thm:wavelet-basedconv} translate into $2\beta>\gamma\ge\beta$, $\beta+2\zeta>\gamma$, $\beta+\zeta\le\alpha+1/2$, $3\beta\le\alpha+1$, and
\[
\|(I-B_{\mathrm{hybrid}}A)v_h\|_{H^1_k(\Omega)}\lesssim k^{\gamma-\zeta+(\alpha-3\beta)/2}\|v_h\|_{H^1_k(\Omega)},\qquad\forall v_h\in V_h.
\]
Hence the error-transfer operator $I-B_{\mathrm{hybrid}}A$ is contractive for large enough $k$ whenever $3\beta+2\zeta>2\gamma+\alpha$.
\end{proof}

\begin{remark}
Given a wavenumber $k$, take $h\approx k^{-2-\varepsilon}$ and the minimal overlap $\delta\approx H$. Then the conditions in Corollary \ref{coro:main} require $H\approx k^{-1-\varepsilon}$ and $\ell\approx(1/2+\varepsilon)\log_2k$. This shows that the resolution scale $2^{-\ell}H$ resolved by the edge coarse basis can be chosen comparable to the data-accuracy fine-grid scale of \cite{wu2014pre}, namely $2^{-\ell}H\ll k^{-3/2}$. Compared with the previous analysis in \cite{FuGongLiWang:2024}, which required at least $\ell>2\log_2k$ levels, the resolution conditions on the coarse mesh size $H$ and the level parameter $\ell$ required here are very mild.

If the minimal overlap is reduced to $\delta\approx k^{-5/4-\varepsilon}$, then the conditions in Corollary \ref{coro:main} require $H\approx k^{-9/8-\varepsilon}$ and $\ell\approx(9/16+\varepsilon)\log_2k$. This means that decreasing the overlap $\delta$ gives smaller local problems, but requires a finer $H$ and a larger $\ell$.
\end{remark}

\section{Numerical experiments}\label{sec:numer}
We examine the convergence of REMs-HS as the wavenumber $k$ increases, beginning with a parameter family satisfying the scaling of Corollary~\ref{coro:main}. We then consider the less restrictive relation $h=k^{-3/2}$ and the more challenging regime of fixed points per wavelength. This progression distinguishes the effect of frequency-dependent trace enrichment from that of outer Krylov acceleration. The three-dimensional experiments extend the comparison to systems with $135,005,697$ degrees of freedom \footnote{Codes are available at \url{https://github.com/aggietx/helmholtz-ibc-reproducibility}}.

\subsection{Discrete setting and iteration counts}
We first take $\Omega=(0,1)^2$, $g=0$, and
\[
 f_h(x)=h^{-2}\exp\!\left(-\frac{|x-x_s|^2}{h^2}\right),
 \qquad x_s=(1/2,1/2).
\]
We use conforming bilinear finite elements on a uniform Cartesian mesh and Gauss quadrature for the element and edge integrals. A uniform coarse partition induces the overlapping subdomains, with $\delta=H/2$ in the two frequency-dependent parameter families and $\delta=H$ in the fixed-ppw experiments. The same subdomains are used for the local corrections and the coarse-basis construction, together with a normalized cubic partition of unity.

Starting from zero, we compare right-preconditioned FGMRES(20) with the unaccelerated iteration $ u_h^{(n+1)}=u_h^{(n)}+B_{\mathrm{hybrid}}(F_h-Au_h^{(n)})$. We report the number $N_{\rm pc}$ of preconditioner applications needed to satisfy $\|Au_h^{(n)}-F_h\|_2\le10^{-6}\|F_h\|_2$. The residual is evaluated directly. In the tables, ``$>200$'' means that the tolerance was not reached within 200 applications, and ``D'' denotes divergence.

\subsection{Asymptotic regime of Corollary \ref{coro:main}}
\label{sec:asymptotic-family}
We choose $\alpha=\beta=\gamma=\zeta=\varepsilon=0.01$ in Corollary~\ref{coro:main}, giving $h\simeq k^{-2-\varepsilon},
 H\simeq\delta\simeq k^{-1-\varepsilon},
 2^{-\ell}\simeq k^{-1/2-\varepsilon}$. For $1/h\in\{64,128,256,512\}$, we set $k=(1/h)^{1/(2+\varepsilon)},
 \ell=\left\lceil(1/2+\varepsilon)\log_2 k\right\rceil$, and choose the coarse partition by $H=2^{\lceil\log_2(h^{-1}k^{-1-\varepsilon})\rceil}h$. The power-of-two rounding preserves the prescribed scaling, while the fine-grid resolution $\mathrm{ppw}:=2\pi/(kh)=2\pi k^{1+\varepsilon}$ increases with the wavenumber.

\begin{table}[htbp]
\centering
\caption{REMs-HS preconditioner application counts for the parameter family of Corollary~\ref{coro:main}, with $\varepsilon=0.01$.}
\label{tab:asymptotic-family}
\setlength{\tabcolsep}{5pt}
\small
\begin{tabular}{c c c c c c c}
\hline
\multirow{2}{*}{$1/h$} & \multirow{2}{*}{$k$} & \multirow{2}{*}{ppw} & \multirow{2}{*}{$H/h$} & \multirow{2}{*}{$\ell$} & \multicolumn{2}{c}{$N_{\rm pc}$} \\
\cline{6-7}
 &  &  &  &  & FGMRES(20) & Unaccelerated \\
\hline
64 & 7.92 & 50.79 & 8 & 2 & 3 & 3 \\
128 & 11.18 & 71.95 & 16 & 2 & 3 & 3 \\
256 & 15.78 & 101.93 & 16 & 3 & 3 & 3 \\
512 & 22.28 & 144.40 & 32 & 3 & 3 & 3 \\
\hline
\end{tabular}
\end{table}

Both unaccelerated REMs-HS and FGMRES(20) converge in three preconditioner applications on every mesh in Table~\ref{tab:asymptotic-family}. The constant counts are consistent with the uniform contraction predicted by Corollary~\ref{coro:main}. In this regime, the local and coarse corrections already provide rapid convergence, and outer Krylov acceleration does not reduce the application count.

\subsection{Frequency-dependent resolution with $h=k^{-3/2}$}
\label{sec:hthreehalves}
To assess convergence at a less restrictive fine-grid scale, we next take $h=k^{-3/2}$ and $\ell=\lceil(1/2+\varepsilon)\log_2 k\rceil$ with $\varepsilon=0.01$. The meshes $1/h\in\{64,128,256,512\}$ then give $k=(1/h)^{2/3}$ and $\mathrm{ppw}=2\pi\sqrt{k}$. We choose the coarse partition by $H=2^{\ell}h$. This sequence retains frequency-dependent trace enrichment while resolving each wavelength with fewer fine elements than the preceding family.

\begin{table}[htbp]
\centering
\caption{REMs-HS preconditioner application counts for $h=k^{-3/2}$ and $\varepsilon=0.01$.}
\label{tab:hthreehalves}
\setlength{\tabcolsep}{5pt}
\small
\begin{tabular}{c c c c c c c}
\hline
\multirow{2}{*}{$1/h$} & \multirow{2}{*}{$k$} & \multirow{2}{*}{ppw} & \multirow{2}{*}{$H/h$} & \multirow{2}{*}{$\ell$} & \multicolumn{2}{c}{$N_{\rm pc}$} \\
\cline{6-7}
 &  &  &  &  & FGMRES(20) & Unaccelerated \\
\hline
64 & 16.00 & 25.13 & 8 & 3 & 2 & 2 \\
128 & 25.40 & 31.67 & 8 & 3 & 2 & 2 \\
256 & 40.32 & 39.90 & 8 & 3 & 2 & 2 \\
512 & 64.00 & 50.27 & 16 & 4 & 2 & 2 \\
\hline
\end{tabular}
\end{table}

Table~\ref{tab:hthreehalves} shows convergence in two preconditioner applications for both iterations throughout the sequence. Thus, with the selected frequency-dependent coarse spaces, rapid convergence persists as $k$ increases from $16$ to $64$ under the weaker relation $h=k^{-3/2}$. The identical counts again show that the unaccelerated Schwarz iteration itself provides the observed reduction.

\subsection{Two-dimensional fixed-ppw regime}

The preceding sequence increases both the points per wavelength and the
trace resolution with $k$.  To assess the method when these quantities remain fixed, we next impose $\mathrm{ppw}:=2\pi/(kh)=20$ and $H=\delta=16h$. Thus $kH=16(2\pi/20)\approx5.03$ is fixed, while the number of wavelengths in
the domain increases under refinement.  The wave resolution and the number
of fine cells in each local problem remain fixed, so the comparison between
$\ell=0$ and $\ell=1$ isolates the influence of the trace resolution at this
order-one value of $kH$.

\begin{table}[htbp]
\centering
\caption{REMs-HS preconditioner application counts in the two-dimensional fixed-ppw regime, with $\delta=H$.}
\label{tab:fixed-ppw}
\setlength{\tabcolsep}{5pt}
\small
\begin{tabular}{c c c c c c}
\hline
\multirow{2}{*}{$1/h$} & \multirow{2}{*}{$k$} & \multicolumn{2}{c}{$\ell=0$} & \multicolumn{2}{c}{$\ell=1$} \\
\cline{3-4}\cline{5-6}
 &  & FGMRES(20) & Unaccelerated & FGMRES(20) & Unaccelerated \\
\hline
64 & 20.11 & 6 & 11 & 4 & 5 \\
128 & 40.21 & 13 & 82 & 5 & 5 \\
256 & 80.42 & 140 & D & 5 & 5 \\
512 & 160.85 & $>200$ & D & 5 & 6 \\
\hline
\end{tabular}
\end{table}

Table~\ref{tab:fixed-ppw} shows that the coarse space $V_{h,0}$ with $\ell=0$ does not control iteration growth at $kH\approx5.03$ that unaccelerated iteration diverges on the two finest meshes, and even FGMRES(20) fails to reach the tolerance within $200$ applications when $1/h=512$.  Increasing the level to $\ell=1$ reduces the unaccelerated counts to five or six throughout the sequence, while FGMRES(20) requires four or five.  The additional trace resolution therefore stabilizes the Schwarz iteration itself. With this coarse space, Krylov acceleration provides only a modest further reduction.

\subsection{Three-dimensional fixed-ppw regime}
Lastly, we take $\Omega=(0,1)^3$, use a discrete point source centered at $(1/2,1/2,1/2)$, and discretize by conforming
trilinear finite elements with Gauss quadrature.  We set $\mathrm{ppw}=20,H=\delta=8h$ and $\ell=0$.

\begin{table}[htbp]
\centering
\caption{REMs-HS preconditioner application counts in the three-dimensional fixed-ppw regime, with $\delta=H$.}
\label{tab:fixed-ppw-3d}
\setlength{\tabcolsep}{5pt}
\small
\begin{tabular}{c c c c c}
\hline
\multirow{2}{*}{$1/h$} & \multirow{2}{*}{$k$} & \multirow{2}{*}{$\dim V_h$} & \multicolumn{2}{c}{$N_{\rm pc}$} \\
\cline{4-5}
 &  &  & FGMRES(20) & Unaccelerated \\
\hline
64 & 20.11 & 274,625 & 4 & 4 \\
128 & 40.21 & 2,146,689 & 4 & 4 \\
256 & 80.42 & 16,974,593 & 5 & 5 \\
512 & 160.85 & 135,005,697 & 5 & 5 \\
\hline
\end{tabular}
\end{table}

Table~\ref{tab:fixed-ppw-3d} shows that both unaccelerated iteration and
FGMRES(20) require four applications on the first two meshes and five on
the remaining two, although the wavenumber increases eightfold and the
number of subdomains grows from $8^3$ to $64^3$.  Thus, at
$kH\approx2.51$, the level-zero space sustains rapid convergence up to
$135,005,697$ fine-grid degrees of freedom without relying on Krylov
acceleration.  In these tests, FGMRES(20) does not reduce the number of
preconditioner applications needed to reach the prescribed tolerance.

\section{Conclusion}\label{sec:conclusion}
 We have designed a novel hybrid Schwarz preconditioner for Helmholtz problems with large wavenumbers, based on an edge multiscale (EMs) coarse space and a one-level RAS-imp preconditioner. The central ingredient is a modified impedance boundary condition (mIBC) imposed both on the subdomain solves and on the local problems generating the EMs coarse space. In contrast to constructions based on the original impedance boundary condition, where the auxiliary problems are indefinite, typically forcing an additional bubble-type enrichment of the coarse space, the mIBC makes every auxiliary local and coarse problem coercive. The proposed method is therefore both stable and cheap. We have rigorously proved the contraction of the associated error-transfer operator under mild conditions on the coarse mesh with $kH\ll 1$ and a markedly relaxed condition on the hierarchical level $\ell>\frac{1}{2}\log_2k$, compared with earlier analyses requiring $\ell>2\log_2 k$. The performance of the method is demonstrated by extensive numerical tests.

\bibliographystyle{abbrv}
\bibliography{reference}
\end{document}